\documentclass[11pt]{amsart}

\usepackage{graphicx}
\usepackage[normalem]{ulem}

\usepackage{xcolor}
\usepackage{cancel}

\usepackage{amsmath,amssymb,amsfonts,graphics}
\usepackage[all]{xy}
\usepackage{enumerate}

\newtheorem{thm}{Theorem}[section]
\newtheorem{cor}[thm]{Corollary}
\newtheorem{lem}[thm]{Lemma}
\newtheorem{prop}[thm]{Proposition}
\theoremstyle{definition}
\newtheorem{defn}[thm]{Definition}

\newtheorem{question}[thm]{Question}
\newtheorem{example}[thm]{Example}
\theoremstyle{remark}
\newtheorem{rem}[thm]{Remark}
\numberwithin{equation}{section}

\newcommand{\R}{\mathbb{R}}
\newcommand{\N}{\mathbb{N}}
\newcommand{\Z}{\mathbb{Z}}

\newcommand{\C}{\mathbb{C}}

\newcommand{\F}{\mathbb{F}}

\newcommand{\tn}{\textnormal}

\newcommand{\pf}{\text{PF}}

\newcommand{\la}{\langle}
\newcommand{\ra}{\rangle}

\newcommand{\lla}{\left\langle}
\newcommand{\rra}{\right\rangle}

\newcommand{\mc}{\mathcal}
\newcommand{\mr}{\mathrm}
\newcommand{\Hi}{\mc H}

\newcommand{\HCp}{\mc H_{\mr C_p}}
\newcommand{\KRq}{\mc K_{\mr R_q}}
\newcommand{\HCq}{\mc H_{\mr C_q}}
\newcommand{\KRp}{\mc K_{\mr R_p}}
\newcommand{\HRp}{\mc H_{\mr R_p}}

\newcommand{\KCp}{\mc K_{\mr C_p}}
\newcommand{\opc}{\otimes_p^{\mr c}}

\newcommand{\opt}{\otimes_p}
\newcommand{\opqc}{\otimes_{p,q}^{\mr c}}

\newcommand{\opq}{\otimes_{p,q}}

\newcommand{\op}{{\mr{op}}}

\begin{document}

\title[]
{New tensor products of C*-algebras and characterization of type I C*-algebras as rigidly symmetric C*-algebras}

\author{Hun Hee Lee}
\address{Department of Mathematical Sciences, Seoul National University, San56-1 Shinrim-dong Kwanak-gu, Seoul 151-747, Republic of Korea}
\email{hunheelee@snu.ac.kr}

\author{Ebrahim Samei}
\address{Department of Mathematics and Statistics, University of Saskatchewan, Saskatoon, Saskatchewan, S7N 5E6, Canada}
\email{samei@math.usask.ca}

\author{Matthew Wiersma}
\address{Department of Mathematics and Statistics, University of Winnipeg, 515 Portage Avenue,
Winnipeg, Manitoba, Canada  R3B 2E9}
\email{m.wiersma@uwinnipeg.ca}


\maketitle


\begin{abstract}

Inspired by recent developments in the theory of Banach and operator algebras of locally compact groups, we construct several new classes of bifunctors $(A,B)\mapsto A\otimes_{\alpha} B$, where $A\otimes_\alpha B$ is a cross norm completion of $A\odot B$ for each pair of C*-algebras $A$ and $B$. For the first class of bifunctors considered $(A,B)\mapsto A{\otimes_p} B$ ($1\leq p\leq\infty$), $A{\otimes_p} B$ is a Banach algebra cross-norm completion of $A\odot B$ constructed in a fashion similar to $p$-pseudofunctions $\text{PF}^*_p(G)$ of a locally compact group. Taking a cue from the recently introduced symmetrized $p$-pseudofunctions due to Liao and Yu and later by the second and the third name authors, we also consider ${\otimes_{p,q}}$ for H\"older conjugate $p,q\in [1,\infty]$ -- a Banach $*$-algebra analogue of the tensor product ${\otimes_{p,q}}$. By taking enveloping C*-algebras of $A{\otimes_{p,q}} B$, we arrive at a third bifunctor $(A,B)\mapsto A{\otimes_{\mathrm C^*_{p,q}}} B$ where the resulting algebra $A{\otimes_{\mathrm C^*_{p,q}}} B$ is a C*-algebra.

For $G_1$ and $G_2$ belonging to a large class of discrete groups, we show that the tensor products $\mathrm C^*_{\mathrm r}(G_1){\otimes_{\mathrm C^*_{p,q}}}\mathrm C^*_{\mathrm r}(G_2)$ coincide with a Brown-Guentner type C*-completion of $\mathrm \ell^1(G_1\times G_2)$ and conclude that if $2\leq p'<p\leq\infty$, then the canonical quotient map $\mathrm C^*_{\mathrm r}(G){\otimes_{\mathrm C^*_{p,q}}}\mathrm C^*_{\mathrm r}(G)\to \mathrm C^*_{\mathrm r}(G){\otimes_{\mathrm C^*_{p,q}}}\mathrm C^*_{\mathrm r}(G)$ is not injective for a large class of non-amenable discrete groups possessing both the rapid decay property and Haagerup's approximation property.

A Banach $*$-algebra $A$ is \emph{symmetric} if the spectrum $\mr{Sp}_A(a^*a)$ is contained in $[0,\infty)$ for every $a\in A$, and \emph{rigidly symmetric} if $A\otimes_{\gamma} B$ is symmetric for every C*-algebra $B$. A theorem of Kugler asserts that every type I C*-algebra is rigidly symmetric. Leveraging our new constructions, we establish the converse of Kugler's theorem by showing for C*-algebras $A$ and $B$ that $A\otimes_{\gamma}B$ is symmetric if and only if $A$ or $B$ is type I. In particular, a C*-algebra is rigidly symmetric if and only if it is type I. This strongly settles a question of Leptin and Poguntke from 1979 and corrects an error in the literature.

\end{abstract}

\section{Introduction}

Let $A$ and $B$ be C*-algebra. Though the maximal and minimal tensor products $A\otimes_{\max}B$ and $A\otimes_{\min}B$ are the canonical C*-completions of the algebraic tensor product $A\odot B$, they need not be the only C*-completions of $A\odot B$. Indeed, the literature contains several different adhoc constructions of C*-completions of $A\odot B$ for special classes of C*-algebras. For example, Effros and Lance study the ``binormal tensor product'' $M\otimes_{\mr{bin}}N$ of two von Neumann algebras $M$ and $N$, and the ``normal tensor product'' $M\otimes_{\mr{nor}} B$ of a von Neumann algebra $M$ and a C*-algebra $B$ in \cite{EL}, and the third author studied intermediate C*-completions of the tensor product of discrete group C*-algebras in \cite{W-tensor}. However, very little is known about finding potentially intermediate C*-completions of $A\odot B$ for \textit{arbitrary} C*-algebras $A$ and $B$. Indeed, only the article \cite{OP} of Ozawa and Pisier has previously been able to exhibit such constructions.

An analogous problem to finding C*-completions of $A\odot B$ for C*-algebras $A$ and $B$ is that of finding \emph{exotic group C*-algebras}, that is C*-completions of the group algebra $\mr L^1(G)$ for a locally compact group $G$ that is intermediate to the full and reduced group C*-algebras $\mr C^*(G)$ and $\mr C^*_{\mr r}(G)$. Thanks in part to its connections with the representation theory of locally compact groups and the revised Baum-Connes conjecture with coefficients, exotic group C*-algebras are much better understood than their C*-tensor product counterparts. The primary examples of potentially exotic group C*-algebra are the \emph{$\mr L^p$ group C*-algebras} $\mr C^*_{\mr L^p}(G)$ introduced by Brown and Guentner in \cite{BG}. Roughly speaking, if $2\leq p\leq \infty$, then $\mr C^*_{\mr L^p}(G)$ is the completion of $\mr L^1(G)$ with respect to the supremum of norms associated to unitary representations $\pi$ of $G$ for which ``many'' matrix coefficients of $\pi$ are $\mr L^p$-integrable. The literature on $\mr L^p$ group C*-algebras $\mr C^*_{\mr L^p}(G)$ is rapidly expanding (see \cite{BR,BEW1,BEW2,dLS,KLQ2,KLQ1,Okayasu,RW,SW-exotic,W-const,W}).
A major breakthrough in the study of $\mr L^p$ group C*-algebras occurred in \cite{SW-exotic}, where a potentially different yet intimately related construction of potentially exotic group C*-algebras appeared. Using the connection between these two potentially different constructions, the second and third authors greatly expanded the class of locally compact groups $G$ for which the C*-algebras $\mr C^*_{\mr L^p}(G)$ are well understood. This potentially different construction considered in \cite{SW-exotic} motivates the main constructions in this paper and, as such, we will now describe it.

If $G$ is a locally compact group and $1\leq p\leq \infty$, then the \textit{$p$-pseudofunctions} $\pf_p(G)$ is the closure of $\mr L^1(G)$ inside of $\mr B(\mr L^p(G))$ where the inclusion of $\mr L^1(G)$ into $\mr B(\mr L^p(G))$ is given by left convolution operators. The $p$-pseudofunctions is a construction of historical significance within the context of Banach algebra theory as well as abstract harmonic analysis. In present day, the Banach algebras $\pf_p(G)$ is actively being studied as an ``$\mr L^p$ operator algebras'' (e.g., see \cite{BP,Gardella,GT,PH}). Recently, Liao and Yu introduced the \emph{symmetrized $p$-pseudofunctions} $\pf_p^*(G)$ (see \cite{LY}), which is the completion of $\mr L^1(G)$ with respect to the norm $\|\cdot\|_{\pf_p^*}$ defined by $\|f\|_{\pf^*_p}:=\max\{\|f\|_{\pf_p},\|f\|_{\pf_q}\}$ where $q$ is the H\"older conjugate of $p$. The symmetrized $p$-pseudofunctions $\pf_p^*(G)$ is much better behaved than $\pf_p(G)$ and, in particular, is a Banach $*$-algebra with respect to the involution coming from $\mr L^1(G)$. The enveloping C*-algebras $\mr C^*(\pf_p^*(G))$ are the classes of potentially exotic group C*-algebras considered by the second and third author in \cite{SW-exotic}. It is unknown whether $\mr C^*(\pf_p^*(G))$ coincides with $\mr C^*_{\mr L^p}(G)$ for all $2\leq p\leq \infty$ and locally compact groups $G$, but equality between these two potentially exotic group C*-algebra constructions has been established for a large class of non-amenable locally compact groups (see \cite{SW-exotic}).

For each $1\leq p\leq \infty$ and every pair of C*-algebras $A$ and $B$, we construct a Banach algebra completion $A{\opt} B$ of $A\odot B$ by considering the closure of $A\odot B$ inside the completely bounded operators on an operator space that can be identified with the Schatten $p$-class operators. Since the Schatten $p$-class operators are a noncommutative $\mr L^p$-space, this completion can be viewed as an analogue of $\pf_p(G)$. In the case when $p=2$, we have that $A{\otimes_2} B\cong A\otimes_{\min}B$ canonically. In the case when $p=1$ or $p=\infty$, the Banach algebras $A{\opt} B$ can be completely described in terms of the Haagerup tensor product. Though we only prove basic properties about the bifunctor $(A,B)\mapsto A{\opt} B$, we believe these Banach algebras will be of independent interest to those studying the $p$-pseudofunctions and will perhaps provide a starting place for studying ``Schatten $p$-class operator algebras''. Taking cues from the introduction of the symmetrized $p$-pseudofunctions, we also consider a ``symmetrized'' variant of ${\opt}$: for each H\"older conjugate pair $p,q\in [1,\infty]$ we obtain a bifunctor $(A,B)\mapsto A{\opq} B$ that assigns to each pair of C*-algebras $A$ and $B$ a Banach $*$-algebra completion of $A\odot B$. By considering the enveloping C*-algebra of $A{\opq} B$, we then arrive at a bifunctor $(A,B)\mapsto A{\otimes_{\mr C^*_{p,q}}} B$ that outputs a C*-completion of $A\odot B$ for all C*-algebras $A$ and $B$.

The C*-algebras $A{\otimes_{\mr C^*_{p,q}}} B$ are currently best understood when $A$ and $B$ are reduced group C*-algebras of certain discrete groups. Indeed, for ``many'' discrete groups with the rapid decay property, the C*-algebras $\mr C^*_{\mr r}(G_1){\otimes_{\mr C^*_{p,q}}} \mr C^*_{\mr r}(G_2)$ coincide with a Brown-Guentner type C*-completion of $\ell^1(G_1\times G_2)$, and we characterize the positive definite functions on $G_1\times G_2$ that extend to a positive linear functional on $\mr C^*_{\mr r}(G_1){\otimes_{\mr C^*_{p,q}}} \mr C^*_{\mr r}(G_2)$ for H\"older conjugate $p,q\in [1,\infty]$ with $q\leq p$ as being those that are ``almost'' the matrix forms of Schatten $p$-class operators. Consequently, the C*-algebras $\mr C^*_{\mr r}(G){\otimes_{\mr C^*_{p,q}}} \mr C^*_{\mr r}(G)$ are pairwise distinct for H\"older conjugate $p,q\in [1,\infty]$ with $q\leq p$ when $G$ belongs to a large class of non-amenable discrete groups possessing both the rapid decay property and the Haagerup approximation property.

So how do these newly introduced tensor products ${\otimes_{\mr C^*_{p,q}}}$ compare to those constructed by Ozawa and Pisier? We can definitively show that for all H\"older conjugate $p,q\in [1,\infty]\backslash \{2\}$, the bifunctor $(A,B)\mapsto A{\otimes_{\mr C^*_{p,q}}}B$ differs from each of the Ozawa-Pisier bifunctors since the Ozawa-Pisier tensor products are injective, whereas the tensor product ${\otimes_{\mr C^*_{p,q}}}$ is not. The classes of C*-algebras $A$ and $B$ for which the $2^{\aleph_0}$ Ozawa-Pisier bifunctors are shown to produce pairwise distinct C*-completions of $A\odot B$ also differs from those considered in this paper. Indeed, the main result in the Ozawa-Pisier paper is that if $M$ and $N$ are von Neumann algebras that are not nuclear (e.g., $\mr B(\ell^2(\N))$), then the bifunctors $(A,B)\mapsto A\otimes_{\alpha} B$ constructed produce $2^{\aleph_0}$ distinct C*-norms on $M\odot B$. In contrast we do not know whether the analogue of this result is true for our tensor products ${\otimes_{\mr C^*_{p,q}}}$. However, we have a very good description of what $\mr C^*_{\mr r}(G_1){\otimes_{\mr C^*_{p,q}}} \mr C^*_{\mr r}(G_2)$ looks like for a large class of discrete groups, whereas it is currently not even known whether the Ozawa-Pisier bifunctors produce any C*-norms that differ from the minimal C*-norm on $\mr C^*_{\mr r}(\F_2)\odot \mr C^*_{\mr r}(\F_2)$.

As an application of the constructions in this paper, we provide a new characterization of type I C*-algebras. A (complex) $*$-algebra $A$ is \textit{symmetric} if for every $a\in A$, the spectrum $\mr{Sp}_A(a^*a)$ of $a^*a$ is contained in $[0,\infty)$. The symmetric condition is extremely important and, in the context of Banach $*$-algebras, has applications to many areas of mathematics including approximation theory, time-frequency analysis and signal processing, non-commutative geometry and geometric group theory (see \cite{GK1}, \cite{GK2}, \cite{GL2}, \cite{black-cuntz}, \cite{rennie}, \cite{Jollisaint}, \cite{Chat} and references therein). A Banach $*$-algebra $A$ is \textit{rigidly symmetric} if the projective tensor product $A\otimes_{\gamma} B$ is symmetric for every C*-algebra $B$. Recently the theory of rigidly symmetric C*-algebras have found applications within the study of single operator theory (see \cite{GKL1,FL,GKL2}). In 1979 Kugler proved that every type I C*-algebra is rigidly symmetric (see \cite{Kugler}), generalizing an earlier result of Leptin stating that $\mr K(\mc H)$ is rigidly symmetric for every Hilbert space $\mc H$ (see \cite{Leptin}).

In their 1979 paper that formalized the notion of rigidly symmetric Banach $*$-algebras (see \cite{LP}), Leptin and Poguntke asked whether every symmetric Banach $*$-algebra is rigidly symmetric and, in particular, whether the C*-algebra $\mr B(\ell^2(\N))$ is rigidly symmetric. The question of whether symmetric Banach $*$-algebras are rigidly symmetric was reiterated in the 1992 paper \cite{Poguntke}. More recently, in 2008 and 2019, this question has been restated in the specific case of $\mr L^1(G)$ where $G$ is a locally compact group  (see \cite{GKL1,FL}), which is the class of Banach $*$-algebras where the condition of being rigidly symmetric is currently finding the most applications. In 2012 Kumar and Rajpal published a paper claiming that every C*-algebra is rigidly symmetric (see \cite[Theorem 2.1]{KR}) which, in particular, would have shown that $\mr B(\mc \ell^2(\N))$ is rigidly symmetric. Unfortunately there is an error in their proof and a remarkably simple counterexample is noted in Example \ref{ex:counterexample} below. We strongly correct the above mentioned error in the literature by establishing the converse of Kugler's theorem. Hence, a C*-algebra is rigidly symmetric if and only if it is type I. Moreover, the following is true.

\begin{thm}[Corollary \ref{cor:Herm-main}]
		If $A$ and $B$ are C*-algebras, then $A\otimes_{\gamma} B$ is symmetric if and only if $A$ or $B$ is type I.
\end{thm}

We finish the introduction by giving an elementary construction of a C*-algebra that is not rigidly symmetric. This appears to give the first example of a symmetric Banach $*$-algebra that is not rigidly symmetric, though it is possible that other counterexamples may have been known to experts due to its simplicity.

\begin{example}\label{ex:counterexample}
	We will show that $\mr C^*_{\mr r}(\F_2)\otimes_{\gamma} \mr C^*_{\mr r}(\F_2)$ is not symmetric. Observe that the diagonal embedding of $\ell^1(\F_2)$ into $\mr C^*_{\mr r}(\F_2)\otimes_\gamma \mr{C^*_r}(\F_2)$ given by
	$$\ell^1(\mathbb F_2)\ni f \mapsto \sum_{s\in \F_2} f(s)\big(\lambda(s)\otimes \lambda(s)\big)\in \mr C^*_{\mr r}(\F_2)\otimes_{\gamma} \mr C^*_{\mr r}(\F_2)$$
	is isometric since the canonical map $$\pi\colon\mr C^*_{\mr r}(\F_2)\otimes_{\gamma} \mr C^*_{\mr r}(\F_2)\to\mr B(\ell^2(\F_2)\otimes_{\gamma} \ell^2(\F_2))$$ is contractive and
	\begin{eqnarray*}
	&&\left\|\pi\bigg(\sum_{s\in \F_2} f(s)\big(\lambda(s)\otimes \lambda(s)\big)\bigg)(\delta_e\otimes \delta_e)\right\|_{\ell^2(\F_2)\otimes_{\gamma} \ell^2(\F_2)}\\
	&=&\left\|\sum_{s\in \F_2} f(s)\big(\delta_s\otimes \delta_s\big)\right\|_{\ell^2(\F_2)\otimes_{\gamma} \ell^2(\F_2)}\\
	&=&\|f\|_{\ell^1}\end{eqnarray*}
	for all $f\in \ell^1(\F_2)$.
	Since $\ell^1(\F_2)$ is non-symmetric (see \cite{palmer}) and the class of symmetric Banach $*$-algebras is closed under the taking of norm-closed $*$-subalgebras, we deduce that $\mr C^*_{\mr r}(\F_2)\otimes_{\gamma} \mr C^*_{\mr r}(\F_2)$ is not symmetric.
\end{example}

\section{Background and notation}

\subsection*{Operator spaces}
The reader will be assumed to be familiar with the basic theory of operator spaces, including the Haagerup tensor product and complex interpolation of operator spaces. The standard references for these topics are \cite{er} and \cite{pisier}.

\subsection*{Notation}

\begin{itemize}
	\item If $X$ is a Banach space, then $X^*$ will denote the dual space of $X$. If $X$ is an operator space, then $X^*$ will denote the operator space dual of $X$.
	\item If $X$ and $Y$ are Banach spaces, then $\mr B(X,Y)$ denotes the space of bounded linear operators from $X$ to $Y$, and $\mr{K}(X,Y)$ denotes the space of compact operators from $X$ to $Y$. If $X$ and $Y$ are operator spaces, then $\mr{CB}(X,Y)$ denotes the space of completely bounded linear maps between $X$ and $Y$.
	\item If $X$ and $Y$ are operator spaces and $T\colon X\to Y$ is a linear operator, we let $\|T\|_{\mr{op}}$ denote the operator norm of $T$ and $\|T\|_{\mr{cb}}$ denote the completely bounded norm {(shortly, cb-norm)} of $T$.
	\item If $\mc H$ and $\mc K$ are Hilbert spaces and $1\leq p\leq \infty$, we let $\mr S_p(\mc H,\mc K)$ denote the space of Schatten $p$-class operators from $\mc H$ to $\mc K$. We will adhere to the convention that $\mr S_\infty(\mc H,\mc K)=\mr K(\mc H,\mc K)$, the space of compact operators from $\mc H$ to $\mc K$ and $\mr S_1(\mc H,\mc K)=\mr{TC}(\mc H,\mc K)$, the trace class operators. We have $\mr{TC}(\mc H,\mc K)^* \cong \mr K(\mc H,\mc K)$ via the duality $\la A, B \ra = \mr{ Tr}(A^TB)$, where $A^T$ is the transpose of $A\in \mr{TC}(\mc H,\mc K)$ and $B\in \mr K(\mc H,\mc K)$. We simply write $\mr S_p(\mc H)$, $\mr{K}(\mc H)$ and $\mr{TC}(\mc H)$ when $\mc H = \mc K$.
	\item If $X$ and $Y$ are Banach spaces and $T\in \mr B(X,Y)$, then $T^\dagger\in \mr B(Y^*,X^*)$ denotes the Banach space adjoint of $T$.
	\item If $\mc H$ and $\mc K$ are Hilbert spaces and $T\in \mr B(\mc H,\mc K)$, then $T^*\in \mr B(\mc K,\mc H)$ denotes the Hilbert space adjoint of $T$.
	
	\item {The Banach space dual $\mc H^*$ of a Hilbert space $\mc H$ can be identified with the complex conjugate $\overline{\mc H}$, where the formal identity map $\mc H \to \overline{\mc H},\; h \mapsto \overline{h}$ is conjugate linear. With this convention we have $T^\dagger (\overline{k}) = \overline{T^* k}$ for $k\in \mc K$ and $T\in \mr B(\mc H,\mc K)$ for another Hilbert space $\mc K$.}
	
	\item Inner products on Hilbert spaces are denoted by $(\cdot, \cdot)$.
	\item If $x\in X$ and $f\in X^*$, we let $\lla x,f\rra=f(x)$. Using the canonical inclusion of $X$ inside of $X^{**}$, we may also write $\lla f,x\rra=f(x)$.
	\item If $X$ and $Y$ are Banach spaces, then $X\otimes_\gamma Y$ denotes the Banach space projective tensor product of $X$ and $Y$. If $X$ and $Y$ are operator spaces of $X$ and $Y$, then $X\otimes_{\mr h} Y$ denotes the Haagerup tensor product, $X\check{\otimes}Y$ denotes the injective operator space tensor product of $X$ and $Y$, $X\hat{\otimes} Y$ denotes the operator space projective tensor product of $X$ and $Y$, and $X^*\otimes_{\mr w^*\mr h}Y^*$ denotes the weak* Haagerup tensor product of $X^*$ and $Y^*$. If $\mc H$ and $\mc K$ are Hilbert spaces, then $\mc H\otimes \mc K$ denotes the Hilbertian tensor product of $\mc H$ and $\mc K$.
	
	\item {The Haagerup tensor product behaves well with respect to complex interpolation of operator space. More precisely, for $0\le \theta \le 1$ and pairs of compatible operator spaces $(X_0, X_1)$ and $(Y_0, Y_1)$ we have $[X_0, X_1]_\theta \otimes_{\mr h} [Y_0, Y_1]_\theta \cong [X_0 \otimes_{\mr h} Y_0, X_1 \otimes_{\mr h} Y_1]_\theta$ completely isometrically via a canonical map.
	}

	\item {We simply write $\ell^2$ for $\ell^2(\N)$.}
	
	\item
	{For each operator space $X$, we define the space of $X$-valued Schatten $p$-class operators from $\mc H$ to $\mc K$ for $1\le p \le \infty$ by
	    $$\mr S_p(\mc H,\mc K;X) := [\mr K(\mc H,\mc K)\check\otimes X,  \mr{TC}(\mc H,\mc K)\hat\otimes X]_{\frac{1}{p}}.$$
	We simply write $\mr S_p(\mc H;X)$ and $\mr S_p(X)$ when $\mc H = \mc K$ and $\mc H = \mc K = \ell^2$, respectively.}

\end{itemize}

\subsection*{Some operator space structures on Hilbert spaces}
We will consider different operator space structures placed upon Hilbert spaces. For a Hilbert space $\mc H$ and $1\leq p\leq \infty$, these include the following:
\begin{itemize}
	\item the \emph{column Hilbert space} $\mc H_{\mr C}:=\mr B(\C,\mc H)$,
	\item the \emph{row Hilbert space} $\mc H_{\mr R}:=\mr B(\overline{\mc H}, \C)$, and
	\item {their complex interpolation spaces
	    $$ \HCp:=[\mc H_{\mr C},\mc H_{\mr R}]_{\frac{1}{p}}\qquad\text{and}\qquad\mc{H}_{\mr R_p}:=[\mc H_{\mr R},\mc H_{\mr C}]_{\frac{1}{p}}.$$
	\item for $p=2$ we have $\mc H_{{\mr C}_2}=\mc H_{{\mr R}_2}$, which we call the \emph{operator Hilbert space}. It is usually denoted by $\mc H_{\mr {oh}}$, which is the unique operator space structure on $\mc H$ such that $\mc H_{\mr {oh}}^*\cong\overline{\mc H}_{\mr {oh}}$ completely isometrically (see \cite[Chapter 7]{pisier} or \cite[Section 3.5]{er}).}
\end{itemize}

{Observe that
$\mc H_{\mr C_\infty} = \mc H_{\mr C}$ and $\mc H_{\mr R_\infty} = \mc H_{\mr R}$. Further, if $q$ denotes the H\"older conjugate of $p$, then $\mc H_{\mr C_p}=\mc H_{\mr R_q}$ and $(\mc H_{\mr C_p})^*\cong (\overline{\mc H})_{\mr R_p}\cong (\overline{\mc H})_{\mr C_q}$ canonically and completely isometrically. The operator spaces $\mc H_{{\mr C_p}}$ and $\mc H_{{\mr R_p}}$ are said to be {\em homogeneous}, which means that any bounded linear map $T: \mc H_{C_p} \to \mc H_{C_p}$ is automatically completely bounded with $\|T\|_{\mr{cb}} = \|T\|_{\mr{op}}$. When $\mc H = \ell^2$ we will simply write $C$, $R$, $C_p$ and $R_p$, $OH$ instead of $\mc H_{\mr C}$, $\mc H_{\mr R}$, $\mc H_{{\mr C_p}}$, $\mc H_{{\mr R_p}}$ and $\mc H_{\mr {oh}}$, respectively.}

\subsection*{Completely contractive Banach algebras}
A \emph{completely contractive Banach algebra} is a Banach algebra $A$ equipped with an operator space structure such that the multiplication map $A\times A\to A$ extends to a completely contractive map $A\hat{\otimes}A\to A$. If $X$ is any (norm complete) operator space, then $\mr{CB}(X)$ is a completely contractive Banach algebra and, hence, every norm-closed subalgebra of $\mr{CB}(X)$ is also a completely contractive Banach algebra. {Some of the} constructions we consider in this paper produce completely contractive Banach algebras, though we will not dwell on this fact.

\subsection*{Opposite C*-algebras}

If $A$ is a C*-algebra and $a\in A$, we let $a^{\mr{op}}$ denote the corresponding element in the opposite algebra $A^{\mr{op}}$. A concrete embedding $A\subset \mr B(\mc H)$ of $A$ gives a corresponding concrete embedding $A^{\mr{op}}\subset \mr B(\overline{\mc H})$ via the identification
$$A^{\op}\ni a^{\op}\leftrightsquigarrow a^\dagger\in \mr B(\overline{\mc H}).$$
This identification will be used without reference throughout the paper.

\section{{The tensor product $\opt$}}

{We begin by defining a new class of tensor products for concrete C*-algebras, which are norm-closed $*$-subalgebras of $\mr B(\mc H)$ for some Hilbert space $\mc H$.}

\begin{defn}\label{defn_rep}
	Let $1\leq p\leq \infty$. For concrete C*-algebras $A\subset \mr B(\mc H)$ and $B\subset \mr B(\mc K)$, consider the representation {$\Pi_p\colon A\odot B\to \mr{CB}(\mc H_{\mr C_p}\otimes_{\mr h}\mc K_{\mr R_p})$} given by
	$$ \Pi_p(a\otimes b)(\xi\otimes \eta)=a\xi\otimes b\eta\qquad$$
	for $a\in A$, $b\in B$, $\xi\in \mc H$, and $\eta\in \mc K$. We let $A\otimes_p^{\mr c} B$ denote {the completely contractive Banach algebra given by} the norm closure of $\Pi_p(A\odot B)$ inside of $\mr{CB}(\mc H_{\mr C_p}\otimes_{\mr h}\mc K_{\mr R_p})$.
\end{defn}

Since the canonical actions of $A$ on $\mc H_{\mr C}$ and $\mc H_{\mr R}$, and $B$ on $\mc K_{\mr C}$ and $\mc K_{\mr R}$ are completely contractive, it follows that the actions of $A$ and $B$ on $\mc H_{\mr C_p}=[\mc H_{\mr C},\mc H_{\mr R}]_{\frac{1}{p}}$ and $\KRp=[\mc K_{\mr R},\mc K_{\mr C}]_{\frac{1}{p}}$ are necessarily also completely contractive. {In particular, the algebra $A\otimes_p^{\mr c} B$ is a cross norm completion of $A\odot B$.}

\begin{rem}
	Let $\mc H$ and $\mc K$ be Hilbert spaces. Recall that
	$$ \mc H_{\mr C_1}\otimes_{\mr h} \mc K_{\mr R_1}=\mc H_{\mr R}\otimes_{\mr h} \mc K_{\mr C}\cong\mr{TC}(\overline{\mc K},\mc H)$$
	and
	$$ \mc H_{\mr C_\infty}\otimes_{\mr h} \mc K_{\mr R_\infty}=\mc H_{\mr C}\otimes_{\mr h} \mc K_{\mr R}\cong\mr{K}(\overline{\mc K},\mc H)$$
	completely isometrically, where we use the identification
	\begin{equation}\label{Eqn:Hilbert-space-identification}\Hi\odot\mc K\ni \sum_{i=1}^n \xi_i\otimes \eta_i \leftrightsquigarrow \sum_{i=1}^n \left( \cdot, \overline{\eta_i}\right)\xi_i\in\mr B(\overline{\mc K},\mc H)\end{equation}
	(see \cite[Proposition 9.3.4]{er}). Hence,
\begin{equation}\label{Eq:Sp-identification}
\HCp\otimes_{\mr h}\KRp \cong[\mc H_{\mr C}\otimes_{\mr h} \mc K_{\mr R}, \mc H_{\mr R}\otimes_{\mr h}\mc K_{\mr C}]_\frac{1}{p}\cong [\mr{K}(\overline{\mc K},\mc H),\mr{TC}(\overline{\mc K},\mc H)]_{\frac{1}{p}}\cong\mr S_p(\overline{\mc K},\mc H)
\end{equation}
	completely isometrically for each $1\leq p\leq \infty$. 
	
	Suppose $A\subset \mr B(\mc H)$ and $B\subset \mr B(\mc K)$ are concrete C*-algebras. Identifying $\HCp\otimes_{\mr h}\KRp$ with $S_p(\overline{\mc K},\mc H)$ as above, $\pi_p$ may be viewed as a representation of $A\odot B$ into $\mr{CB}(\mr S_p(\overline{\mc K},\mc H))$ where
	\begin{equation}\label{eq:action-on-Sp}
	\Pi_p(a\otimes b)T=aTb^{\dagger}
	\end{equation}
	for all $a\in \mr B(\mc H)$, $b\in \mr B(\mc K)$, and $T\in \mr S_p(\overline{\mc K},\mc H)$. As such, we may alternatively view $A\opc B$ as being a norm closed subalgebra of {$\mr{CB}(\mr S_p(\overline{\mc K},\mc H)).$}	
	The identification of $\mc H_{\mr C_p}\otimes_{\mr h} \mc K_{\mr R_p}$ with $\mr S_p(\overline{\mc K},\mc H)$, and the corresponding form of $\Pi_p$ described by Equation \eqref{eq:action-on-Sp} will be used repeatedly throughout the paper.
\end{rem}

	As alluded to in the introduction, the above remark allows one to think of the tensor product $A\opc B\subset \mr{CB}(\mr S_p(\overline{\mc K},\mc H))$ as being an analogue of the Banach algebra of $p$-pseudofunctions $\pf_p(G)\subset \mr B(\mr L^p(G))$ for a locally compact group $G$ since the Schatten $p$-classes are non-commutative $\mr L^p$-spaces.

	{We can completely describe the operator space structure of $A\otimes_p^{\rm c} B$ in the cases when $p=1$ and $p=\infty$.}	
{\begin{prop}\label{prop:p=1 or p=infty}
	Let $A\subset \mr B(\mc H)$ and $B\subset \mr B(\mc K)$ be concrete C*-algebras.
	\begin{enumerate}
		\item The map $J_\infty\colon A\otimes_\infty^{\mr c}B \to A\otimes_{\mr h}B^{\op}$ defined on elementary tensors by $J_\infty(a\otimes b)=a\otimes b^{\op}$ is a complete isometry.
		\item The map $J_1\colon A\otimes_1^{\mr c} B\to B^\op\otimes_{\mr h} A$ defined on elementary tensors by $ J_1(a\otimes b)=b^{\op}\otimes a$ is a complete isometry.
	\end{enumerate}
\end{prop}
\begin{proof}By \cite[Proposition 2.1 and Corollary 2.3]{BS},
	$$\mr{CB}(\mr K(\overline{\mc K},\mc H),\mr B(\overline{\mc K},\mc H))\cong \mr B(\mc H)\otimes_{\mr w^*\mr h} \mr B(\overline{\mc K})$$
	completely isometrically and
	$$\mr{CB}(\mr{TC}(\overline{\mc K},\mc H))\cong \mr B(\overline{\mc K})\otimes_{\mr w^*\mr h} \mr B(\mc H)$$
	completely isometrically. Tracing through the relevant maps shows that $J_\infty$ and $J_1$ are completely isometric since the (weak*) Haagerup tensor product is injective.
\end{proof}
\begin{rem}\label{rem:oss}
	If $B$ is any C*-algebra, then Proposition \ref{prop:p=1 or p=infty} implies we have completely isometric identifications $\C\otimes_1 B\cong\C\otimes_{\infty} B\cong B^{\mr{op}}$.
	Since the operator space structure on $B^{\mr{op}}$ can differ significantly from that on $B$ (e.g., if $\mc H$ is an infinite dimensional Hilbert space, then the maps $\mr K(\mc H)\ni a\mapsto a^{\op} \in\mr K(\mc H)^{\mr{op}}$ and $\mr K(\mc H)^{\mr{op}}\ni a^{\op}\mapsto a \in \mr K(\mc H)$ fail to be completely bounded), the operator space structure of $A\opc B$ will be of lesser interest to us than the norm structure.
\end{rem}}
{In the case when $p=2$, our construction coincides isometrically (but not necessarily completely isometrically) with the minimal tensor product.}
{\begin{prop}\label{prop:p=2}
	If $A\subset \mr B(\mc H)$ and $B\subset \mr B(\mc K)$ are concrete C*-algebras, then
	$$ A\otimes_2^{\mr c} B\cong A\otimes_{\min}B$$
	isometrically.
\end{prop}
\begin{proof} {Recall that $\mc H_{\mr C_2} \cong \mc H_{\mr R_2} \cong \mc H_{\mr{oh}}$ and $\mc H_{\mr{oh}}\otimes_{\mr h}\mc K_{\mr{oh}}\cong(\mc H\otimes \mc K)_{\mr{oh}}$ completely isometrically (see \cite[Corollary 7.3]{pisier}). Since operator Hilbert spaces are homogeneous we get the conclusion we wanted.}
\end{proof}
The following simple observations will occasionally be used to shorten proofs within this paper.
\begin{lem}\label{lem:easier}
	Let $A$ and $B$ be C*-algebras. For $*$-representations $\pi\colon A\to \mr B(\mc H)$ and $\sigma\colon B\to \mr B(\mc K)$ we consider their amplifications $\tilde{\pi}\colon A\to \mr B(\mc H \otimes \ell^2),\; a\mapsto \pi(a) \otimes I_{\ell^2}$ and $\tilde{\sigma}\colon B\to \mr B(\mc K \otimes \ell^2),\; b\mapsto \sigma(b) \otimes I_{\ell^2}$. Let
	$\Pi_p\colon A\odot B\to \mr{CB}(S_p(\overline{\mc K},\mc H))$ and $\Pi_p'\colon A\odot B\to \mr {CB}(S_p(\overline{\mc K \otimes \ell^2},\mc H \otimes \ell^2))$ be the canonical representations associated to $(\pi(A), \sigma(B))$ and $(\tilde{\pi}(A), \tilde{\sigma}(B))$, respectively, as in Definition \ref{defn_rep}. Then $\|\Pi_p(x)\|_{\mr{cb}}=\|\Pi_p'(x)\|_{\op}$ for every $x\in A\odot B$.
\end{lem}
\begin{proof}
		{Recall that for an operator space $X$ and $T\in \mr CB(X)$ we have
		$$\|T\|_{\mr cb} = \|I_{S_p} \otimes T:S_p(X) \to S_p(X)\|_{\op}.$$
		See \cite[Theorem 9.2.3]{pisier} for example. Here, we are using the convention $S_p = S_p(\overline{\ell^2}, \ell^2)$. Since we have $S_p(X) \cong \mr C_p \otimes_{\mr h} X \otimes_{\mr h} \mr R_p$ completely isometrically in a canonical way we can easily see that $S_p(S_p(\overline{\mc K},\mc H)) \cong S_p(\overline{\mc K \otimes \ell^2},\mc H \otimes \ell^2)$ completely isometrically also in a canonical way. This gives us the wanted conclusion.}
\end{proof}}
	
	{We can define the analog of our tensor product $\otimes_p^{\rm c}$ for arbitrary C*-algebras by considering faithful embeddings. For simplicity, we will only consider the Banach algebraic structure of this constructions, although the analogously defined operator space structure is of independent interest.	
\begin{defn}\label{D:independent}
	Let $A$ and $B$ be C*-algebras with faithful $*$-representations $\pi\colon A\to \tn B(\mc H)$ and $\sigma\colon B\to \tn B(\mc K)$. We will let $A\otimes_p B$ denote the Banach algebraic completion of $A\odot B$ that one obtains by considering $A\odot B$ as a subset of $\pi(A)\otimes_p^{\rm c} \sigma(B)$.
\end{defn}
One should be concerned about whether this construction depends on the choice of embeddings. The authors are grateful to N. Ozawa, L. Turowska and an anonymous referee who independently suggested the following argument that shows it does not. This answers a question asked in an earlier version of this paper.
	\begin{thm}\label{thm:independence}
		Let $A$ and $B$ be C*-algebras and $1\leq p\leq \infty$. If $\pi_i\colon A\to \tn B(\mc H_i)$ and $\sigma_i\colon B\to \tn B(\mc K_i)$ are non-degenerate, injective $*$-representations ($i=1,2$), then
		$$ \pi_1(A)\otimes_{p}^{\tn c}\sigma_1(B)\cong \pi_2(A)\otimes_{p}^{\tn c}\sigma_2(B)$$
		isometrically.
	\end{thm}
	\begin{proof}
		We will first assume that $A,B,\mc H_1,\mc H_2, \mc K_1,\mc K_2$ are each separable. We will also replace $\pi_1$ and $\pi_2$ with their infinite amplifications, which allow us to consider operator norms in place of cb norms and invoke Voiculescu's theorem implying that there exist unitaries $\{u_n\}\subset \tn B(\mc H_1,\mc H_2)$, $\{v_n\}\subset\tn B(\mc K_1,\mc K_2)$ so that
		$$\|\pi_1(a)-u_n^*\pi_2(a)u_n\|\to 0\quad\text{and}\quad\|\sigma_1(b)-v_n^*\sigma_2(b)v_n\|\to 0 \quad \text{as}\quad n\to \infty$$
		for all $a\in A$, $b\in B$. Let $\sum_{i=1}^k a_i\otimes b_i\in A\odot B$. Then
		$$\left\|\sum_{i=1}^k u_n^*\pi_2(a_i)u_n\otimes v_n^*\sigma_2(b_i)v_n\right\|_{\op}=\left\|\sum_{i=1}^k \pi_2(a_i)\otimes \sigma_2(b_i)\right\|_{\op}$$
		for each $n\in \N$. Since
		$$\sum_{i=1}^k u_n^*\pi_2(a_i)u_n\otimes v_n^*\sigma_2(b_i)v_n\to\sum_{i=1}^k \pi_1(a_i)\otimes \sigma_1(b_i)\quad \text{as}\quad n\to \infty$$
		in norm, we deduce that
		$$\left\|\sum_{i=1}^k \pi_1(a_i)\otimes \sigma_1(b_i)\right\|_{\op}=\left\|\sum_{i=1}^k \pi_2(a_i)\otimes \sigma_2(b_i)\right\|_{\op}.$$
		This gives the desired result under appropriate assumptions of separability.\\		
		\indent We next assume that $A$ and $B$ are separable and $\pi\colon A\to\rm B(\mc H)$ and $\sigma\colon B\to \tn B(\mc K)$ are faithful representations where $\mc H$ and $\mc K$ are not necessarily separable. Choose decompositions $\mc H=\bigoplus_{i\in I} \mc H_i$ and $\mc K=\bigoplus_{j\in J} \mc K_j$ where each $\mc H_i$ is a separable $A$-invariant subspace of $\mc H$ and each $\mc K_j$ is a separable $B$-invariant subspace of $\mc K$. For each  countable subset $\alpha$ of $I$ and $\beta$ of $J$ (where finite sets are considered countable), let $P_\alpha$ be the orthogonal projection from $\mc H$ onto $\bigoplus_{i\in \alpha} \mc H_i$ and $Q_\beta$ be the orthogonal projection from $\mc K$ onto $\bigoplus_{j\in \beta} \mc K_j$. Then $P_\alpha\otimes Q_\beta$ is a complete contraction belonging
		to $\tn {\rm CB}(\mc H_{{\rm C}_p}\otimes_{\rm h} \mc K_{{\rm R}_q})$ for all $\alpha$ and $\beta$ since $\mc H_{{\rm C}_p}$ and $\mc K_{{\rm R}_q}$ are homogeneous operator spaces. Let $x\in A\odot B$. So
		$$\left\|(P_{\alpha}\otimes Q_\beta)(\pi\otimes\sigma (x))(P_{\alpha}\otimes Q_\beta)\right\|_{\text{cb}}\leq \left\|\pi\otimes\sigma (x)\right\|_{\text{cb}}$$
		for all $\alpha$ and $\beta$. As
		$$(P_{\alpha}\otimes Q_\beta)(\pi\otimes\sigma (x))(P_{\alpha}\otimes Q_\beta)\to\pi\otimes\sigma (x)$$
		in the SOT, we deduce that
		$$\left\|(P_{\alpha}\otimes Q_\beta)(\pi\otimes\sigma (x))(P_{\alpha}\otimes Q_\beta)\right\|_{\text{cb}}\to_{\alpha,\beta}\left\|\pi\otimes\sigma (x)\right\|_{\text{cb}}.$$
		Since $P_{\alpha}\pi(\cdot)P_\alpha$ and $Q_{\beta}\sigma(\cdot)Q_\beta$ are faithful representations of $A$ and $B$ when $\alpha$ and $\beta$ are ``large enough'' countable subsets of $I$ and $J$, we deduce that
		$$\left\|\big[\pi\otimes\sigma (x_{i,j})\big]\right\|_{\text{cb}}=\left\|\big[(P_{\alpha}\otimes Q_\beta)(\pi\otimes\sigma (x_{i,j}))(P_{\alpha}\otimes Q_\beta)\big]\right\|_{\text{cb}}$$
		when $\alpha\subset I$ and $\beta\subset J$ are ``large enough'' countable subsets by the result obtained in the previous paragraph. Hence, the desired result is proved for when $A$ and $B$ are separable but $\mc H_1,\mc H_2,\mc K_1,\mc K_2$ are not necessarily separable.\\
		\indent Now suppose $A$ and $B$ are arbitrary C*-algebras that are not necessarily separable. Let $x=\sum_{i=1}^k a_i\otimes b_i\in A\odot B$. Considering the separable C*-subalgebras of $A$ and $B$ generated by $a_1,\ldots,a_k$ and $b_1,\ldots,b_k$, respectively, allows us to deduce that
		$$\big\|\pi_1\otimes\sigma_1(x)\big\|_{\text{cb}}= \big\|\pi_2\otimes\sigma_2(x)\big\|_{\text{cb}}.$$
	\end{proof}}

We now show the bifunctor $(A,B)\mapsto A{\opt} B$ has nice functorial properties. In particular, we show this bifunctor to be isometrically commutative, functorial with respect to completely contractive maps, and injective.
{\begin{prop}\label{prop:containment}
	Let $1\leq p\leq \infty$. The mapping $(A,B)\mapsto A\opt B$ is injective in the sense that if $A$ and $B$ are C*-algebras containing C*-subalgebras $A_0$ and $B_0$, then the canonical map $A_0\opt B_0\to A\opt B$ is isometric.
\end{prop}
\begin{proof}
	This is immediate from the definition of $\opt$ and Theorem \ref{thm:independence}.
\end{proof}
\begin{prop}\label{prop:commutative}
	Let $A$ and $B$ be C*-algebras and $1\leq p\leq \infty$. The mapping $(A,B)\mapsto A\opt B$ is isometrically commutative in the sense that the tensor flip map $\Sigma\colon A\odot B\to B\odot A$ extends to a surjective isometry $A\opt B\cong B\opt A$.
\end{prop}}

\begin{proof}
	{It suffices to verify this in the case when $A\subset \tn B(\mc H)$ and $B\subset \tn B(\mc H)$ are concrete C*-algebras. In this case, we will identify $A\opt B$ with $A\opc B$.} {Consider the canonical representations $\Pi_p\colon A\odot B\to \mr B(\HCp\otimes_{\mr h} \KRp)$ and $\widetilde{\Pi}_p\colon B\odot A\to \mr B(\KCp\otimes_{\mr h} \HRp)$, respectively. Applying Lemma \ref{lem:easier}, we may replace $\mc H$ and $\mc K$ with $\mc H\otimes \ell^2$ and $\mc K\otimes \ell^2$ for amplified representations $\Pi'_p$ and $\widetilde{\Pi}'_p$, respectively, and we get $\|\pi'_p(x)\|_{\op}=\|\pi_p(x)\|_{\mr{cb}}$ and $\|\widetilde{\Pi}'_p(y)\|_{\op}=\|\widetilde{\Pi}_p(y)\|_{\mr{cb}}$ for all $x\in A\odot B$ and $y\in B\odot A$. This allows us to focus on comparing operator norms rather than cb-norms.}
	
	Identify $\HCp\otimes_{\mr h} \KRp$ and $\KCp\otimes_{\mr h} \HRp$ with $\mr S_p(\overline{\mc K},\mc H)$ and $\mr S_p(\overline{\mc H},\mc K)$, respectively. The Banach space adjoint maps $\varphi\colon \mr S_p(\overline{\mc K},\mc H)\to \mr S_p(\overline{\mc H},\mc K)$ and $\psi\colon \mr S_p(\overline{\mc H},\mc K)\to\mr S_p(\overline{\mc K},\mc H)$  given by $\varphi(T)=T^\dagger$ and $\psi(S)=S^\dagger$ for $T\in \mr S_p(\overline{\mc K},\mc H)$ and $S\in \mr S_p(\overline{\mc H},\mc K)$ {are onto isometries and, hence, for all $x\in A\odot B$, we have that
	$$\|\Pi_p(x)\|_{\op}=\|\varphi\circ \Pi_p(x)\circ \psi\|_{\op}.$$
	Observe that if $a\in A$, $b\in B$ and $S\in \mr S_p(\overline{\mc H},\mc K)$, then
	\begin{eqnarray*}
	&& \varphi\circ \Pi_p(a\otimes b)\circ \psi(S)\\
	&=& \varphi\circ \Pi_p(a\otimes b) S^\dagger\\
	&=& \varphi(aS^\dagger b^\dagger)\\
	&=& bSa^\dagger\\
	&=& \widetilde{\Pi}_p(b\otimes a)S.
	\end{eqnarray*}}
	It follows that the tensor flip map $\Sigma$ extends to an isometric isomorphism $A{\opt} B\cong B{\opt} A$.
\end{proof}

\begin{rem}
	Let $B$ be any noncommutative {concrete} C*-algebra and $A$ a non-zero {concrete} C*-algebra. Since the map $B\ni b\mapsto b^{\op}\in B^{\op}$ is not completely isometric, we have that $A{\otimes_\infty^{\rm c}} B\cong A\otimes_{\mr h} B^{\op}$ is not canonically completely isometrically isomorphic to $B{\otimes_\infty^{\rm c}} A\cong B\otimes_{\mr h} A^{\op}$. In particular, {the analog of} Proposition \ref{prop:commutative} becomes false when the term ``isometrically'' is replaced with ``completely isometrically'' {when studying the operator space structure of $A\opc B$}.
\end{rem}
{\begin{prop}\label{prop:functorial}
	Let $1\leq p\leq \infty$. If $A$, $A'$, $B$, $B'$ are C*-algebras, and $T\colon A\to A'$ and $S\colon B\to B'$ are complete contractions, then the map $T\otimes S\colon A\odot B\to A'\odot B'$ extends to a contraction
	$$T\otimes S\colon A\opt B\to A'\opt B'.$$
\end{prop}}
\begin{proof}
	Let $\pi'\colon A'\to \mr B(\mc H')$ and $\sigma'\colon B'\to\mr B(\mc K')$ be {faithful} $*$-representations of $A'$ and $B'$. By the structure theorem for completely bounded maps (see \cite[Theorem 1.6]{pisier}), there exists $*$-representations $\pi\colon A\to \mr B(\mc H)$ and $\sigma\colon B\to \mr B(\mc K)$ and contractive operators $V_1, V_2\in \mr B(\mc H',\mc H)$, $W_1,W_2\in \mr B(\mc K', \mc K)$ such that
	$$ \pi'(T(a))=V_2^*\pi(a)V_1\qquad\text{and}\qquad \sigma'(S(b))=W_2^*\sigma(b)W_1$$
	for all $a\in A$ and $b\in B$. {By ``enlarging'' the $*$-representations $\pi$ and $\sigma$ if necessary, we may assume that $\pi$ and $\sigma$ are faithful.} {Since the operator spaces $\HCp$ and $\HRp$ are homogeneous, we can easily conclude that the maps $U_1:=V_1\otimes W_1\colon \HCp'\otimes_{\mr h} \KRp'\to  \HCp\otimes_{\mr h} \KRp$ and $U_2:=V_2^*\otimes W_2^*\colon \HCp\otimes_{\mr h} \KRp\to  \HCp'\otimes_{\mr h} \KRp'$ are completely contractive.} So $R\colon \mr{CB}(\HCp\otimes_{\mr h} \KRp)\to \mr{CB}(\HCp'\otimes_{\mr h} \KRp')$ given by
	$$ R(\,\cdot\,):=U_2(\,\cdot\,)U_1$$
	is completely contractive. {Since $T\otimes S$ is the restriction of $R$ to $\pi(A) \opc \sigma(B)$, we deduce that $T\otimes S\colon A{\opc} B\to \pi'(A')\opc \sigma'(B')$ is completely contractive. This, in particular, implies that
	$$ T\otimes S\colon A\opt B\to A'\opt B'$$
	is contractive.}
\end{proof}

\section{Symmetrized tensor products and their enveloping C*-algebras}

Recall that we may consider the tensor products $A\opc B$ and $A{\opt} B$ as being analogues of the $p$-pseudofunction $\mr{PF}_p(G)$. Taking cues from the recently introduced symmetrized $p$-pseudofunctions due to Liao and Yu (see \cite{LY}), and the breakthroughs that this construction has spawned, we next consider a symmetrized version of the tensor products introduced in the previous section. Let us begin by making the following simple observations.

\begin{prop}
	Suppose $A\subset \mr B(\mc H)$ and $B\subset \mr B(\mc K)$ are concrete C*-algebras. If $1\leq p\leq \infty$
and $q$ is the H\"older conjugate of $p$, then
		$$\Pi_p(a\otimes b)^\dagger=\Pi_q(a^\dagger \otimes b^\dagger ) \in \mr {CB}((\overline{\mc H})_{\mr C_q}\otimes_{\mr h} (\overline{\mc K})_{\mr R_q})=\mr{CB}((\HCp\otimes_{\mr h} \KRp)^*).$$
	for all $a\in A$, $b\in B$.
\end{prop}

\begin{proof}
	Let $a\in A\subset \mr B(\mc H)$, $b\in B\subset \mr B(\mc K)$, $\xi,\xi'\in \mc H$, and $\eta,\eta'\in \mc K$. We will view $\xi\otimes \eta$ as an element of $\HCp\otimes_{\mr h} \KRq$ and $\overline{\xi'}\otimes \overline{\eta'}$ as an element of $(\overline{\mc H})_{\mr C_q}\otimes_{\mr h} (\overline{\mc K})_{\mr R_q}=(\HCp\otimes_{\mr h} \KRq)^*$. Then
	\begin{eqnarray*}
		&&\lla \Pi_p(a\otimes b)(\xi\otimes \eta),\overline{\xi'\otimes\eta'}\rra\\
		&=&(a\xi,\xi')(b\eta,\eta')\\
		&=&(\xi,a^*\xi')(\eta,b^*\eta')\\
		&=&\lla \xi\otimes \eta, \overline{a^*\xi'} \otimes \overline{b^*\eta'} \rra\\
		&=&\lla \xi\otimes \eta, \Pi_q(a^\dagger \otimes b^\dagger)(\overline{\xi'} \otimes \overline{\eta'})\rra.
	\end{eqnarray*}
	It follows by linearity that
	$$ \Pi_p(a\otimes b)^\dagger=\Pi_q(a^\dagger\otimes b^\dagger).$$
\end{proof}

\begin{cor}\label{cor:*-isometric}
	{Let $A\subset \mr B(\mc H)$ and $B\subset \mr B(\mc K)$ be concrete C*-algebras, and suppose $p$ and $q$ are a H\"older conjugate pair.}
	\begin{enumerate}
		\item The map $J\colon A\odot B\to A^\op\odot B^\op$ defined on elementary tensors by $J(a\otimes b)=a^\op\otimes b^\op$ extends to an isometric surjection $A\opc B\cong A^{\op}\otimes_q^{\rm c} B^\op$.
		\item If $x\in A\odot B$, {then we have $\overline{\Pi_p(x^*)} = \Pi_q(x)^\dagger$, so that  $\|\Pi_p(x^*)\|_{\mr{op}}=\|\Pi_q(x)\|_{\mr{op}}$. Consequently, we have $\|x^*\|_{A{\opt} B}=\|x\|_{A{\otimes_{q}}B}$. Here, for $T\in \mr{CB}(\HCp \otimes_{\mr h} \KRp)$, the operator $\overline{T} \in \mr {CB}((\overline{\mc H})_{\mr C_p}\otimes_{\mr h} (\overline{\mc K})_{\mr R_p})$ is given by $\overline{T}( \overline{\xi} \otimes \overline{\eta}):= \overline{T(\xi \otimes \eta)}$, $\xi\in \mc H$, $\eta \in \mc K$.}
	\end{enumerate}
\end{cor}

\begin{proof}
	\emph{(1)}: Immediate from the previous lemma {and the identification
$A^{\op}\ni a^{\op}\leftrightsquigarrow a^\dagger\in \mr B(\overline{\mc H}).$}
	
	\emph{(2)}:
	 The equality $\overline{\Pi_p(x^*)} = \Pi_q(x)^\dagger$ is straightforward from the definition. For the norm comparison we observe the following. If the operator $X \in \mr S_p(\overline{\mc K},\mc H)$ corresponds to $\xi \otimes \eta\in \HCp \otimes_{\mr h} \KRp$, then the operator $(X^*)^\dagger \in \mr S_p(\mc K, \overline{\mc H})$ corresponds to $\overline{\xi} \otimes \overline{\eta} \in (\overline{\mc H})_{\mr C_p}\otimes_{\mr h} (\overline{\mc K})_{\mr R_p}$. Since both of the Banach space adjoint and the Hilbert space adjoint have the same $S_p$-norms as the original operator, we can conclude that $\|\overline{T}\|_{\mr{op}} = \|T\|_{\mr{op}}$. This explains $\|\Pi_p(x^*)\|_{\mr{op}}=\|\Pi_q(x)\|_{\mr{op}}$.
	
		{The conclusion for the universal version comes from the same trick of replacing $\mc H$ and $\mc K$ with $\mc H\otimes \ell^2$ and $\mc K\otimes \ell^2$, respectively, and appealing to Lemma \ref{lem:easier}.}
	
\end{proof}

%
%
%
%

Equipped with the above observations, we now define ``symmetrized'' versions of the tensor products $\otimes_p^{\mr c}$ and {$\opt$}.

\begin{defn}
	Suppose $A$ and $B$ are concrete C*-algebras and $1\leq p\leq \infty$ and let $q$ be the H\"older conjugate of $p$. We define $A \otimes_{p,q}^{\mr c} B$ to be, in the language of interpolation theory, the intersection space $(A\opc B)\cap (A\otimes_q^{\mr c}B)$. That is, $A{\opqc} B$ is the operator space completion of $A\odot B$ with respect the norms $\|\cdot\|_{\mr M_n(A \otimes_{p,q}^{\mr c} B)}$ on $\mr M_n(A\odot B)$ given by
	$$ \|x\|_{\mr M_n(A \otimes_{p,q}^{\mr c} B)}:=\max\big\{\|x\|_{\mr M_n(A\otimes_p^{\mr c} B)},\|x\|_{\mr M_n(A\otimes_q^{\mr c} B)}\big\}$$
	{Similarly, we define $ A\otimes_{p,q} B$ to be the (Banach space) intersection space $(A\opt B)\cap (A\otimes_qB)$
	for C*-algebras $A$ and $B$, where $p,q\in [1,\infty]$ are H\"older conjugate.}
\end{defn}

\begin{rem}
	Let $A$ and $B$ be concrete C*-algebras (resp., C*-algebras) and $p,q\in [1,\infty]$ be H\"older conjugate.
	\begin{enumerate}
		\item By definition, we have $A\otimes_{q,p}^{\mr c}B= A\opqc B$ {(resp., $A\otimes_{q,p}B = A\opq B$).}
		\item {$A{\opq} B$ is} a Banach $*$-algebra by Corollary \ref{cor:*-isometric} (2).
	\end{enumerate}
\end{rem}

%
%

A convenient feature of the symmetrized tensor products $\otimes_{{p,q}}^{\mr c}$ and {$\otimes_{{p,q}}$} is that there is nice containment relation between $A\opqc B$ and $A{\opq} B$ as the H\"older conjugate pair $p,q$ varies. To show this we will use the following remark, which is separated from the rest of the argument since it will also be used elsewhere.

\begin{rem}\label{rem:norm-relation}
	Let $A\subset \mr B(\mc H)$ and $B\subset \mr B(\mc K)$ be concrete C*-algebras. If $p_0,p_1\in [1,\infty]$, $\theta\in (0,1)$ and $\frac{1}{p}=\frac{1-\theta}{p_0}+\frac{\theta}{p_1}$, then
	\begin{equation}\label{norm-relation}
	\big\|[x_{i,j}]\big\|_{\mr M_n(A\opc B)}\leq \big\|[x_{i,j}]\big\|_{\mr M_n(A\otimes_{p_0}^{\mr c}B)}^{1-\theta}\big\|[x_{i,j}]\big\|_{\mr M_n(A\otimes_{p_1}^{\mr c}B)}^{\theta}
	\end{equation}
	for all $n\in \N$ and $[x_{i,j}]\in \mr M_n(A\odot B)$ since
	$$\begin{array}{c l}
	& \mc H_{\mr C_{p}}\otimes_{\mr h}\mc K_{\mr R_{p}}\\
	\cong &\big[\mc H_{\mr C_{p_0}},\mc H_{\mr C_{p_1}}\big]_\theta\otimes_{\mr h} \big[\mc K_{\mr R_{p_0}},\mc K_{\mr R_{p_1}}\big]_\theta\\
	\cong&\big[\mc H_{\mr C_{p_0}}\otimes_{\mr h} \mc K_{\mr R_{p_0}},\mc H_{\mr C_{p_1}}\otimes_{\mr h} \mc K_{\mr R_{p_1}}\big]_\theta
	\end{array} $$
	completely isometrically. {The inequality \eqref{norm-relation} also implies that if $A$ and $B$ are (abstract) C*-algebras and $x\in A\odot B$, then $\|x\|_{A\opt B}\leq \|x\|_{A\otimes_{p_0}B}^{1-\theta}\|x\|_{A\otimes_{p_1}B}^{\theta}.$}
\end{rem}

\begin{prop}\label{prop:inclusions}
	Let $A\subset \mr B(\mc H)$ and $B\subset \mr B(\mc K)$ be concrete C*-algebras (resp., $A$ and $B$ be C*-algebras). Suppose $p,q\in [1,\infty]$ are H\"older conjugate, and either $p<p'<q$ or $q<p'<p$. Let $q'$ be the H\"older conjugate of $p'$. The identity map on $A\odot B$ extends to an injective complete contraction $\sigma\colon A\opqc B\to A\otimes_{p',q'}^{\mr c}B$ (resp., {injective contraction $\sigma\colon A\opq B\to A\otimes_{p',q'}B$}).
\end{prop}

\begin{proof}
	It suffices to prove this in the case of concrete C*-algebras.
	Choose $\theta\in (0,1)$ so that $\frac{1}{p'}=\frac{1-\theta}{p}+\frac{\theta}{q}$.  Then
	$$\begin{array}{c l}
	&\big\|[x_{i,j}]\big\|_{\mr M_n(A\otimes_{p'}^{\mr c}B)}\\
	\leq& \big\|[x_{i,j}]\big\|_{\mr M_n(A\otimes_{p}^{\mr c}B)}^{1-\theta}\big\|[x_{i,j}]\big\|_{\mr M_n(A\otimes_{q}^{\mr c}B)}^\theta\\
	\leq& \big\|[x_{i,j}]\big\|_{\mr M_n(A\otimes_{p,q}^{\mr c}B)}.
	\end{array}$$
	for all $[x_{i,j}]\in \mr M_n(A\odot B)$ by the above remark.
	A similar argument shows that $$\big\|[x_{i,j}]\big\|_{\mr M_n(A\otimes_{q'}^{\mr c}B)}\leq \big\|[x_{i,j}]\big\|_{\mr M_n(A\otimes_{p,q}^{\mr c}B)}$$
	and, thus,
	$$\big\|[x_{i,j}]\big\|_{\mr M_n(A\otimes_{p',q'}^{\mr c}B)}\leq \big\|[x_{i,j}]\big\|_{\mr M_n(A\otimes_{p,q}^{\mr c}B)}$$
	for all $[x_{i,j}]\in \mr M_n(A\odot B)$. So the identity map on $A\odot B$ extends to a complete contraction $\sigma\colon A\otimes_{p,q}^{\mr c} B\to A\otimes_{p',q'}^{\mr c}B$.
	
	We now show $\sigma$ is injective. Suppose $\sigma(x)=0$ for some $x\in A\opqc B$ and choose a sequence $\{x_n\}\subset A\odot B$ converging to $x$ in norm. For $p\leq s\leq q$, we let $\sigma_s\colon A\opc B\to A\otimes_s^{\mr c}B$ be the canonical contraction. We also let $\iota_{s}\colon \mc H\odot \mc K\to \mc H_{\mr C_s}\otimes_{\mr h} \mc K_{\mr R_s}$ be the natural inclusion for $1\leq s\leq \infty$. Then for all $y,y'\in \mc H\odot \mc K$,
	$$ \lla \sigma_p(x)\iota_p(y),\overline{\iota_q(y')}\rra=\lim_{n\to \infty}\lla \sigma_p(x_n)\iota_p(y),\overline{\iota_q(y')}\rra=\lim_{n\to \infty}\lla \sigma_{p'}(x_n)\iota_{p'}(y),\overline{\iota_{q'}(y')}\rra=0.$$
	So $\sigma_p(x)=0$. Similarly $\sigma_q(x)=0$ and, so, $x=0$.
\end{proof}

\begin{rem}\label{rem:*-semisimple}
	Let $p$ and $q$ be a H\"older conjugate pair. Taking $p'=2$ in the above proposition implies the identity map on $A\odot B$ extends to an injective $*$-homomorphism $A{\opq} B\to A\otimes_{\min} B$ for every $1\leq p\leq \infty$. In particular, $A{\otimes_{p,q}} B$ is \emph{$*$-semisimple}, i.e., $*$-representations separate points of $A{\opq} B$. Equivalently, the canonical map from $A{\otimes_{p,q}} B$ into its enveloping C*-algebra is injective.
\end{rem}

As a consequence of the previous proposition, we find that Proposition \ref{prop:p=2} does not hold completely isometrically.

\begin{rem}
	Consider the tensor product $\C\otimes_{1,\infty}\mr K(\ell^2(\N))$. By Proposition \ref{prop:p=1 or p=infty}, we know the map $a^\op\mapsto 1\otimes a$ extends to a completely isometric injection $\mr K(\ell^2(\N))^{\op}\to\C\otimes_{1,\infty}\mr K(\ell^2(\N))$. So, by the previous proposition, the map
	$$\mr K(\ell^2(\N))^\op\ni a^\op\mapsto 1\otimes a\in \mr \C\otimes_2^{\mr c} \mr K(\ell^2(\N))$$
	is a complete contraction. Since the map $\mr K(\ell^2(\N))^\op\ni a^\op\mapsto a\in \mr K(\ell^2(\N))$ is not completely bounded, we deduce the canonical mapping $\C\otimes_2 \mr K(\ell^2(\N))\to \mr K(\ell^2(\N))$ is not completely bounded. In particular, this shows that Proposition \ref{prop:p=2} is false when the term ``isometric'' is replaced with ``completely isometric''.
\end{rem}

We pause to briefly contrast Proposition \ref{prop:inclusions} with work of Choi (see \cite{Choi}) to show how different the norm structures of $A{\opt} B$ and $A{\opq} B$ can be.

\begin{rem}
	Incidentally, the norm structure of $A\otimes_{\mr h} B^{\op}\cong A\otimes_{\infty} B$ has been studied by Choi in the unpublished manuscript \cite{Choi} where it is shown that there exists nuclear C*-algebras $A$ and $B$ such that the map $a\otimes b^\op\mapsto a\otimes b$ does not extend to a bounded linear map $A\otimes_{\mr h} B^\op \to A\otimes_{\min} B$ (see \cite[Proposition 3.2 and Remark 3.3]{Choi}). In other words, there exists nuclear C*-algebras $A$ and $B$ such that the identity map on $A\odot B$ does not extend to a bounded linear map $A\otimes_\infty B\to A{\otimes_2} B$.
	
	Fix C*-algebras $A$ and $B$ such that the identity map on $A\odot B$ cannot extend to a bounded linear map $A\otimes_\infty B\to A\otimes_{\min} B$. Suppose $2<p<\infty$ and choose $\theta\in (0,1)$ such that $\frac{1}{p}=\frac{1-\theta}{2}$. Remark \ref{rem:norm-relation} implies that
	$$ \|x\|_{A{\opt} B}\leq \|x\|_{A\otimes_{\min}B}^{1-\theta}\|x\|_{A\otimes_{\infty}B}^\theta$$
	for all $x\in A\odot B$ and, hence, the identity map on $A\odot B$ does not extend to a bounded linear map $A{\opt} B\to A\otimes_{\min}B$. A similar argument shows there to exist nuclear C*-algebras $C$ and $D$ such that the identity map on $C\odot D$ does not extend to a bounded linear map $C{\otimes_q}D\to C\otimes_{\min}D$ for all $1\leq q<{2}$.
	
	Since the identity map on $A\odot B$ extends to a (complete) contraction $A{\opq} B\to A\otimes_{\min} B$ for all C*-algebras $A$ and $B$ and H\"older conjugate pair $p,q\in [1,\infty]$, we deduce that if $p\in [1,\infty]\backslash\{2\}$, then there exists nuclear C*-algebras $A$ and $B$ such that the identity map on $A\odot B$ does not extend to a bounded linear map $A{\opt} B\to A{\opq} B$, where $q$ is the H\"older conjugate of $p$. By Corollary \ref{cor:*-isometric} (2) we equivalently deduce that for every $p\in [1,\infty]\backslash\{2\}$, there exists nuclear C*-algebras $A$ and $B$ such that the involution on $A\odot B$ does not extend to a bounded conjugate linear map on $A{\opt} B$.
\end{rem}

\subsection*{Enveloping C*-algebras}

The tensor product functors we have constructed thus far do not produce C*-algebras in general, but there is a way of constructing potentially new C*-tensor products from the symmetrized construction: simply take the enveloping C*-algebras {of $A\opq B$}. This can be thought of an analogue of the potentially exotic C*-algebras $\mr C^*(\pf^*_p(G))$, the enveloping C*-algebra of $\pf^*_p(G)$, considered by the second and third author in \cite{SW-exotic}.

\begin{defn}\label{def:enveloping}
	{For C*-algebras $A$ and $B$, we set
	$$ A\otimes_{\mr C^*_{p,q}}B:=\mr C^*(A\otimes_{p,q}B),$$
	the enveloping C*-algebra of $A\opt B$.}
\end{defn}

Observe that if $A$ and $B$ are C*-algebras, then $A\otimes_{\mr C^*_{2,2}}B\cong A\otimes_{\min}B$ isometrically by Proposition \ref{prop:p=2}. Viewing the enveloping C*-algebras of the symmetrized $p$-pseudofunctions as being analgous to the constructions in Definition \ref{def:enveloping} and the full group C*-algebra construction as being analagous to the maximal C*-tensor product construction, one might expect that $\otimes_{\mr C^*_{1,\infty}}$ coincides with $\otimes_{\max}$, the maximal C*-tensor product, since $\mr C^*(\mr{PF}^*_{1}(G))\cong\mr C^*(\mr L^1(G))\cong\mr C^*(G)$ for every locally compact group $G$. We finish this section by proving that this is indeed the case.

\begin{thm}\label{thm:max}
	If $A$ and $B$ are C*-algebras, then
	$ A\otimes_{\mr C^*_{1,\infty}}B\cong A\otimes_{\max}B$
	canonically.
\end{thm}

\begin{proof} Recall that the map identity map on $A\odot B$ extends to a (complete) contraction $A\otimes_{\mr h} B\to A\otimes_{\max}B$ since $A\otimes_{\mr h} B$ embeds canonically inside of the free product $A*B$ and $A*B$ quotients onto $A\otimes_{\max} B$. As $(A\otimes_{\max} B)^\op\cong A^\op\otimes_{\max} B^\op$ canonically, we also have that the map $A^\op\otimes_{\mr h} B^\op\to A\otimes_{\max} B$ defined on elementary tensors by $a^\op\otimes b^\op\mapsto a\otimes b$ is contractive.
Consider the complex interpolation space $[A\otimes_{\mr h} B,A^\op\otimes_{\mr h} B^\op]_{1/2}$. By the above observations, the canonical map
$$[A\otimes_{\mr h} B,A^\op\otimes_{\mr h} B^\op]_{1/2}\to A\otimes_{\max} B$$
is contractive. Notice that
$$\begin{array}{c l}
	&[A\otimes_{\mr h} B,A^\op\otimes_{\mr h} B^\op]_{1/2}\\
	\cong& [A,A^\op]_{1/2}\otimes_{\mr h} [B, B^\op]_{1/2}\\
	\cong& [A,A^\op]_{1/2}\otimes_{\mr h} [B^\op, B]_{1/2}\\
	\cong & [A\otimes_{\mr h} B^\op,A^\op\otimes_{\mr h} B]_{1/2}\\
	\cong & [A\otimes_{\infty}B, B\otimes_1 A]_{1/2}
\end{array}$$
(completely) isometrically, where the final line is given by Proposition \ref{prop:p=1 or p=infty}. So, by the above observation and Proposition \ref{prop:commutative}, if $\Sigma\colon A\odot B\to B\odot A$ is the tensor flip and $x\in A\odot B$, then
$$\begin{array}{r c l}
\|x\|_{A\otimes_{\max}B} & \leq &\|x\|_{[A\otimes_{\mr h} B,A^\op\otimes_{\mr h} B^\op]_{1/2}}\\
&\leq& \|x\|_{A\otimes_{\infty}B}^{1/2}\|\Sigma x\|_{B\otimes_1 A}^{1/2}\\
	&=& \|x\|_{A\otimes_{\infty}B}^{1/2}\|x\|_{A\otimes_{1}B}^{1/2}\\
	&\leq & \max\{\|x\|_{A\otimes_{\infty}B},\|x\|_{A\otimes_{1}B}\}\\
	&=& \|x\|_{A\otimes_{1,\infty}B}.
\end{array}$$
\end{proof}

\section{Example: Group C*-algebras}

In this section, we will study the {tensor product $\otimes_{\mr C^*_{p,q}}$ in the context of} reduced group C*-algebras of discrete groups. We will see that the C*-algebras $\mr C^*_{\mr r}(G_1)\otimes_{\mr C^*_{p,q}}\mr C^*_{\mr r}(G_2)$ coincide with a Brown-Guentner type C*-completion of $\ell^1(G_1\times G_2)$ for many such groups. Let us begin by reviewing the Brown-Guentner construction.

Let $G$ be a locally compact group. For each unitary representation of $\pi\colon G\to \mr B(\mc H_\pi)$ of $G$ and $\xi,\eta\in \mc H_\pi$, we let $\pi_{\xi,\eta}\colon G\to \C$ denote the matrix coefficient function defined by $\pi_{\xi,\eta}(s)=(\pi(s)\xi,\eta)$. Recall that the \emph{Fourier-Stieltjes algebra} of $G$ is the set of all matrix coefficient functions of $G$
$$ \mr B(G):=\{\pi_{\xi,\eta} \mid \pi\colon G\to \mr B(\mc H_\pi)\text{ is a unitary representation of }G,\ \xi,\eta\in \mc H_\pi\}.$$
Then $\mr B(G)=\mr C^*(G)^*$ with respect to the dual pairing
$$ \lla a,u\rra=(\pi(a)\xi,\eta)$$
for all $a\in \mr C^*(G)$ and $u\in \mr B(G)$, where $\pi\colon G\to \mr B(\mc H_\pi)$ and $\xi,\eta\in \mr B(\mc H_\pi)$ are such that $u=\pi_{\xi,\eta}$. Then $\mr B(G)$ is a Banach algebra with respect to pointwise addition and  multiplication, and the norm it obtains by virtue of being the dual space of $\mr C^*(G)$. We also let $\mr P(G)$ denote the cone of continuous positive definite functions on $G$, i.e.
$$ \mr P(G):=\{\pi_{\xi,\xi} \mid \pi\colon G\to \mr B(\mc H_\pi)\text{ is a unitary representation of }G,\ \xi\in \mc H_\pi\}.$$

All facts about the Fourier-Stieltjes algebra used in this paper may be found in either \cite{Ey} or \cite{KL}.

\begin{defn}[Brown-Guentner {\cite{BG}}]
	Let $G$ be a discrete group and $D$ be a set of functions from $G$ to $\mathbb{C}$. A unitary representation $\pi\colon G\to \mr B(\mc H_\pi)$ is a \emph{$D$-representation} of $G$ if $\mc H_\pi$ admits a dense subspace $\mc H_0$ such that $\pi_{\xi,\xi}\in D$ for all $\xi\in \mc H_0$. The \emph{$D$-C*-algebra} of $G$ is the ``completion'' of $\ell^1(G)$ with respect to the C*-seminorm
	$$\|f\|_{\mr C^*_D}:=\sup\{\|\pi(f)\|: \pi\text{ is a $D$-representation of }G\}$$
	for $f\in \ell^1(G)$.
\end{defn}

The Brown-Guentner construction has previously essentially only been studied in the cases when $D=\ell^p(G)$ and $D=\mr c_0(G)$. We will study in the case when $G$ is a direct product of two discrete groups and $D$ is identified with matrix coefficients of Schatten $p$-class operators.

Suppose $G_1$ and $G_2$ are discrete groups. We let $ \mr S_p(G_1\times G_2)$ denote the set of all functions $f\colon G_1\times G_2\to \C$ such that $[f(s,t)]_{(s,t)\in G_1\times G_2}$
is the matrix representation of an operator in $\mr S_p(\overline{\ell^2(G_2)},\ell^2(G_1))$ with respect to the standard bases of $\ell^2(G_1)$ and $\overline{\ell^2(G_2)}$. Further, we identify $\mr S_p(G_1\times G_2)$ with $\mr S_p(\overline{\ell^2(G_2)},\ell^2(G_1))$ in the natural way. We will write $\mr C^*_{\mr S_p}(G_1\times G_2)$ in place of {$\mr C^*_{\mr S_p(G_1 \times G_2)}(G_1\times G_2)$} in order simplify notation.
Then
\begin{equation}\label{eq:BSp}
\mr B_{\mr S_p}(G_1\times G_2):=\overline{\mr{span}\big(\mr P(G_1\times G_2)\cap \mr S_p(G_1\times G_2)\big)}^{\sigma(\mr B(G_1\times G_2),\mr C^*(G_1\times G_2)}
\end{equation}
is the dual of $\mr C^*_{\mr S_p}(G_1\times G_2)$ by \cite[Proposition 4.2 and Proposition 4.3]{W} since $\mr S_p(G_1\times G_2)$ is invariant under left and right translation by $G_1\times G_2$.


\begin{rem}
	As noted by Brown and Guentner (see \cite{BG}), the left regular representation $\lambda$ of a discrete group $G$ is a $\mr{c_c}(G)$-representation since $\lambda_{\xi,\xi}\in \mr{c_c}(G)$ for all $\xi\in \mr{c_c}(G)\subset \ell^2(G)$. {Here, $\mr{c_c}(G)$ is the space of all finitely supported functions on $G$.} Suppose $1\leq p\leq 2$ and $G_1,G_2$ are discrete groups. Since
	$$\mr{c_c}(G_1\times G_2)\subset \mr S_p(G_1\times G_2)\subset \ell^2(G_1\times G_2)$$
	and $\mr C^*_{\ell^2}(G)\cong \mr C^*_{\mr r}(G)$ canonically for every discrete group $G$ (see \cite{BG}), we therefore conclude that $\mr C^*_{\mr S_p}(G_1\times G_2)\cong \mr C^*_{\mr r}(G_1\times G_2)\cong \mr C^*_{\mr r}(G_1)\otimes_{\min}\mr C^*_{\mr r}(G_2)$ canonically. We will soon see that this need not be true for $p>2$.
%
\end{rem}

Our study of the C*-algebras $\mr C^*_{\mr S_p}(G_1\times G_2)$ begins with showing that $\mr B_{\mr S_p}(G_1\times G_2)$ is an ideal of $\mr B(G_1\times G_2)$.

\begin{prop}
	Let $u\in \mr B(G_1\times G_2)$ and $2\leq p\leq \infty$. If $f\in \mr S_p(G_1\times G_2)$, then $uf\in \mr S_p(G_1\times G_2)$. Further, the multiplication map
	$$\mr S_p(G_1\times G_2)\ni f\mapsto uf\in \mr S_p(G_1\times G_2)$$
	is completed bounded with {cb-norm} at most $\|u\|_{\mr B(G_1\times G_2)}$.
\end{prop}

\begin{proof}
{Since
	$$\mr S_p(G_1\times G_2)\cong \big[\mr S_2(G_1,G_2),\mr S_\infty(G_1\times G_2)\big]_{\frac{2}{p}}$$
as operator spaces, it is enough to check the cases $p=2$ and $p=\infty$. For the case $p=\infty$ we choose a unitary representation $\pi\colon G_1\times G_2\to \mr B(\Hi_\pi)$ and vectors $\xi,\eta\in \Hi_\pi$ so that $u(s,t)=( \pi(s,t)\xi,\eta)$ and {$\|u\|_{\mr B(G_1\times G_2)}=\|\xi\|\|\eta\|$.} Defining $\xi_s=\pi(s,e)\xi$ for $s\in G_1$ and $\eta_t=\pi(e,t^{-1})\xi$ for $t\in G_2$, we have $u(s,t)=(\xi_s,\eta_t)$ and $\|\xi_s\|\|\eta_t\|=\|u\|_{\mr B(G_1\times G_2)}$ for all $(s,t)\in G_1\times G_2$. Now the structure theorem for completely bounded Schur multipliers tells us that the corresponding map has cb-norm $\le \|u\|_{\mr B(G_1\times G_2)}$. When $p=2$, we recall the fact that operator Hilbert spaces are homogeneous, so that the corresponding multiplier has cb-norm at most $\|u\|_{\ell^\infty(G_1\times G_2)}\leq \|u\|_{\mr B(G_1\times G_2)}$.}
\end{proof}

\begin{cor}
	If $2\leq p\leq \infty$, then $\mr B_{\mr S_p}(G_1\times G_2)$ is an ideal of $\mr B(G_1\times G_2)$.
\end{cor}

\begin{proof}
	The previous proposition and Equation \eqref{eq:BSp} imply $\mr B(G_1\times G_2)\cap \mr{S}_p(G_1\times G_2)$ is an ideal of $\mr B(G_1\times G_2)$. This corollary now follows from the fact that multiplication in the Fourier-Stieltjes algebra is separately weak*-weak* continuous.
\end{proof}


We next show that if $2\leq p\leq \infty$, then identity map on $\mr{c_c}(G_1\times G_2)\cong \mr{c_c}(G_1)\odot {\mr c_c}(G_2)$ extends to a contraction $\mr{C^*_r}(G_1){\opq}\mr{C^*_r}(G_2)\to \mr{C}^*_{\mr S_p}(G_1\times G_2)$.

\begin{lem}\label{lem:convolution-bound}
	Let $f\in \mr{c_c}(G_1\times G_2)$ and $2\leq p\leq \infty$. Then
	$$ \|f\|_{\mr C^*_{\mr S_p}}\leq \liminf_{n\to\infty} \|(f^* * f)^{*n}\|_{\mr S_q}^{\frac{1}{2n}}$$
	where $q$ is the H\"older conjugate of $p$.
\end{lem}

\begin{proof}
	Suppose $\pi$ is a representation of $G_1\times G_2$ admitting a dense subspace $\mc H_0$ of $\mc H_\pi$ so that $\pi_{\xi,\eta}\in \mr S_p(G_1,G_2)$ for all $\xi,\eta\in \mc H_0$. In the proof of \cite[Theorem 1]{CHH}, it is shown by Cowling, Haagerup and Howe that
	$$\|\pi(f)\| =\sup_{\xi\in \mc H_0}\lim_{n\to\infty}\left( \pi(f^**f)^n\xi,\xi\right)^{\frac{1}{2n}}.$$
	So we deduce
	\begin{eqnarray*}
		\|\pi(f)\| 	&=&\sup_{\xi\in \mc H_0}\lim_{n\to \infty} \left(\sum_{s\in G_1, t\in G_2} (f^**f)^{*n}(s,t)\pi_{\xi,\xi}(s,t)\right)^{\frac{1}{2n}}\\
		&=& \sup_{\xi\in \mc H_0}\lim_{n\to \infty} \mr{Tr}\bigg(\big[(f^**f)^{*n}(s,t)\big]^T\big[\pi_{\xi,\xi}(s,t)\big]\bigg)^{\frac{1}{2n}}\\
		&\leq & \sup_{\xi\in \mc H_0}\liminf_{n\to \infty} \|(f^**f)^{*n}\|_{\mr S_q}^{\frac{1}{2n}}\|\pi_{\xi,\xi}\|_{\mr S_p}^{\frac{1}{2n}}\\
		&=& \liminf_{n\to \infty} \|(f^**f)^{*n}\|_{\mr S_q}^{\frac{1}{2n}}.
	\end{eqnarray*}
{Here, $X^T$ refers to the transpose of the matrix $X$.}	
\end{proof}

{Let us move our attention back to the pair of C*-algebras $(\mr C^*_{\mr r}(G_1), \mr C^*_{\mr r}(G_2))$. Note that if $f=\sum_{i=1}^n a_i \delta_{s_i}\otimes \delta_{t_i}\in \mr{c_c}(G_1\times G_2) \subseteq \mr C^*_{\mr r}(G_1) \odot \mr C^*_{\mr r}(G_2)$, $\xi\in \ell^2(G_1)$, and $\eta\in \ell^2(G_2)$, then
$$\Pi_p(f)(\xi\otimes \eta)=\sum_{i=1}^n a_i (\delta_{s_i}* \xi)\otimes (\delta_{t_i}*\eta).$$
Hence, identifying $\mr S_p(G_1\times G_2)$ with $\ell^2(G_1)_{\mr C_p}\otimes_{\mr h} \ell^2(G_2)_{\mr R_p}\cong\mr S_p(\overline{\ell^2(G_2)},\ell^2(G_1))$, we have
\begin{equation}\label{eq:pi-p-group}
\Pi_p(f)g = f*g
\end{equation}
for all $f\in \ell^1(G)$ and $g\in \mr S_p(G_1\times G_2)$.}

\begin{cor}\label{cor:p,q-to-Sp}
	Let $2\leq p\leq \infty$ and consider the concrete C*-algebras $\mr C^*_{\mr r}(G_1)\subset \mr B(\ell^2(G_1))$ and $\mr C^*_{\mr r}(G_2)\subset \mr B(\ell^2(G_2))$. The identity map on $\mr{c_c}(G_1\times G_2)$ extends to a contractive $*$-homomorphism
	$$\mr C^*_{\mr r}(G_1){\opq}\mr C^*_{\mr r}(G_2)\to \mr C^*_{\mr S_p}(G_1\times G_2).$$
\end{cor}

\begin{proof}
	Let $f\in\mr{c_c}(G_1\times G_2)$, $2\leq p\leq \infty$ and suppose $q$ is the H\"older conjugate of $p$. Lemma \ref{lem:convolution-bound} and Equation \eqref{eq:pi-p-group} imply
	$$\begin{array}{r c l}
	&&\|f\|_{\mr C^*_{\mr S_p}}\\
	&\leq & \liminf_{n\to \infty}\|(f^**f)^n\|^{\frac{1}{2n}}_{\mr S_q}\\
	&\leq &\liminf_{n\to \infty}\|\Pi_q(f^**f)^{n-1}(f^**f)\|_{\mr S_q}^{\frac{1}{2n}}\\
	&\leq & \liminf_{n\to \infty}\|\Pi_q(f^**f)^{n-1}\|^{\frac{1}{2n}}\|f^**f\|_{\mr S_q}^{\frac{1}{2n}}\\
	&\leq & \|\Pi_q(f^*)\pi_q(f)\|^{\frac{1}{2}}\\
	&\leq & \max\{\|\Pi_q (f^*)\|,\|\Pi_q(f)\|\}\\
	&=& \max\{\|\Pi_p (f)\|,\|\Pi_q(f)\|\}\\
	&{\leq}& \|f\|_{\mr C^*_{\mr r}(G_1){\opq}\mr C^*_{\mr r}(G_2)}.
	\end{array}$$
\end{proof}

Corollary \ref{cor:p,q-to-Sp} establishes a strong relationship between $\mr C^*_{\mr r}(G_1){\otimes_{\mr C^*_{p,q}}} \mr C^*_{\mr r}(G_2)$ and $\mr C^*_{\mr S_p}(G_1\times G_2)$: the identity map on $\mr {c_c}(G_1\times G_2)$ extends to a $*$-homomorphism
$$\mr C^*_{\mr r}(G_1){\otimes_{\mr C^*_{p,q}}} \mr C^*_{\mr r}(G_2)\to \mr C^*_{\mr S_p}(G_1\times G_2).$$
We do not know whether this $*$-homomorphism is injective in general.
\begin{question}
	Suppose $G_1$ and $G_2$ are discrete groups and $p,q\in [1,\infty]$ are H\"older conjugate with $p\geq 2$. Is it necessarily true that
	$$\mr C^*_{\mr r}(G_1){\otimes_{\mr C^*_{p,q}}} \mr C^*_{\mr r}(G_2)\cong \mr C^*_{\mr S_p}(G_1\times G_2)$$
	canonically?
\end{question}
This question will be shown to have a positive solution for many discrete groups with the rapid decay property later in this section. {Recall that a map $\mc L: G \to \mathbb{R}_+$ is called a {\em length function} on a discrete group $G$ if:
    \begin{itemize}
        \item[(i)] $\mc L(gh) \le \mc L(g) + \mc L(h)$, $g,h\in G$;
        \item[(ii)] $\mc L(g) = \mc L(g^{-1})$, $g\in G$;
        \item[(iii)] $\mc L(1) = 0$, where 1 denotes the identity of $G$.
    \end{itemize}
When $G$ is {\em finitely generated}, there is a canonical length function called the {\em word length function} (for a fixed set of generators of $G$), which is usually denoted by $|\cdot|$. We say that $G$ has a {\em rapid decay property} if there is a length function $\mc L$ on $G$ such that the space $H^\infty_{\mc L}$ is contained in the space $\tn C^*_r(G)$, where $H^\infty_{\mc L} = \{f:G \to \mathbb{C}\,|\, \sum_{g\in G}|f(g)|^2\mc L(g)^s <\infty\; \forall s\in \mathbb{R}\}$.
}

{Towards the above goal,} we next introduce spaces that can be interpreted as ``weighted Schatten classes''. Let $G_1$ and $G_2$ be discrete groups. For every non-negative function $\omega\colon G_1\times G_2\to (0,\infty)$ and $1\leq p\leq \infty$, we set
$$ \mr S_p(G_1\times G_2,\omega)=\{f\colon G_1\times G_2\to \C \mid f\omega\in \mr S_p(G_1\times G_2)\}.$$
The spaces $\mr S_p(G_1\times G_2,\omega)$ equipped with the norm $\|\cdot\|_{\mr S_p,\omega}$ defined by $\|f\|_{\mr S_p,\omega}=\|f\omega\|_{\mr S_p}$ can be thought of as an analogue of weighted $\mr L^p$-spaces and may be of independent interest for study. In this paper, however, these spaces will mostly be used for notational convenience.

\begin{lem}\label{lem:Sq-inclusion}
	Suppose $G_1$ and $G_2$ are discrete groups possessing the rapid decay property with respect to length functions $\mc L_1$ and $\mc L_2$, respectively. For $0<t<\infty$ and $i=1,2$, let $\varphi_{i,t}\colon G_i\to (0,1]$ be the function defined by $\varphi_{i,t}(s)=e^{-t\mc L_i(s)}$ for $s\in G_i$. Then
	\begin{equation}\label{eq:Sq-inclusion}
		\mr S_q(G_1\times G_2,\varphi_{1,t}^{-1}\times \varphi_{2,t}^{-1})\subset \mr C^*_{\mr r}(G_1){\opq} \mr C^*_{\mr r}(G_2)
	\end{equation}
	for all H\"older conjugate $p,q\in [1,\infty]$ with $q\leq p$ and for all $t>0$.
\end{lem}

\begin{proof}
	We first show this inclusion in the case when $q=2$.
	
	Since $G_1$ and $G_2$ have the rapid decay property with respect to $\mc L_1$ and $\mc L_2$, respectively, $G_1\times G_2$ has the rapid decay property with respect to the length function $\mc L\colon G\to [0,\infty)$ defined by $\mc L(s,t)=\mc L_1(s)+\mc L_2(t)$ (see \cite{Jollisaint}). So there exists a $d>0$ such that if $\omega_d\colon G_1\times G_2\to [1,\infty)$ is the polynomial weight defined by $\omega_d(s,t)=\big(1+\mc L_1(s)+\mc L_2(t)\big)^d$, then $\ell^2(G_1\times G_2, \omega_d)\subset \mr C^*_{\mr r}(G_1\times G_2)$. Since $\varphi_{1,t}^{-1}\times \varphi_{2,t}^{-1}$ grows faster than $\omega_d$, we deduce that
	$$\begin{array}{c l}
	&\mr S_2(G_1\times G_2, \varphi_{1,t}^{-1}\times \varphi_{2,t}^{-1})\\
	= & \ell^2(G_1\times G_2, \varphi_{1,t}^{-1}\times \varphi_{2,t}^{-1})\\
	\subset & \ell^2(G_1\times G_2, \omega_d)\\
	\subset & \mr C^*_{\mr r}(G_1\times G_2)\\
	= & \mr C^*_{\mr r}(G_1)\otimes_{\min} \mr C^*_{\mr r}(G_2)\\
	=& \mr C^*_{\mr r}(G_1){\otimes_{{2,2}}}\mr C^*_{\mr r}(G_2).
	\end{array}$$
	This establishes the result in the case when $q=2$.
	
	We now address the case when $q=1$. A similar argument as used above shows that we obtain contractive embeddings $\ell^2(G_1,\varphi_{1,t}^{-1})\subset \mr C^*_{\mr r}(G)$ and $\ell^2(G_2,\varphi_{2,t}^{-1})\subset \mr C^*_{\mr r}(G_2)$. So the universality of the Banach space projective tensor product guarantees a contractive embedding
	$$ \ell^2(G_1,\varphi_{1,t}^{-1})\otimes_\gamma \ell^2(G_2,\varphi_{2,t}^{-1})\subset \mr C^*_{\mr r}(G_1)\otimes_{1,\infty} \mr C^*_{\mr r}(G_2).$$
	As
	$$\ell^2(G_1,\varphi_{1,t}^{-1})\otimes_\gamma \ell^2(G_2,\varphi_{2,t}^{-1})\cong\mr S_1(\overline{\ell^2(G_2,\varphi_{2,t}^{-1})},\ell^2(G_1,\varphi_{1,t}^{-1}))\cong\mr S_1(G_1\times G_2,\varphi_{1,t}^{-1}\times \varphi_{2,t}^{-1})$$
	canonically, we have established the result in the case when $q=1$.
	
	{Let $\pi\colon \mr C^*_{\mr r}(G_1)\to \mr B(\mc H)$ and $\sigma\colon \mr C^*_{\mr r}(G_2)\to \mr B(\mc K)$ be faithful $*$-representations (such as the left regular representations).
	By Lemma \ref{lem:easier}, replacing $\pi$ and $\sigma$} with their infinite amplifications, it suffices to consider the operator norm in place of the completely bounded norms. Notice that establishing Equation \ref{eq:Sq-inclusion} in the case when $q=2$ is equivalent to stating the bilinear map
	\begin{equation}\label{eq:Sq-inclusion-lem}
	\mr{c_c}(G_1\times G_2)\times \big(\mc H\odot \mc K\big)\ni (f,x)\mapsto \sum_{(s,t)\in G_1\times G_2}f(s,t)\big(\pi(s)\otimes \sigma(t)\big)(x)\in \mr \mc H\odot \mc K
	\end{equation}
	extends to a bounded bilinear map
	$$\mr S_2(G_1\times G_2,\varphi_{1,t}^{-1}\times\varphi_{2,t}^{-1})\times\big(\mc H_{\mr C_2}\otimes_{\mr h} \mc K_{\mr R_2}\big)\to \mc H_{\mr C_2}\otimes_{\mr h} \mc K_{\mr R_2}.$$
	Similarly, establishing Equation \ref{eq:Sq-inclusion} in the case when $q=1$ is equivalent to stating that Equation  \eqref{eq:Sq-inclusion-lem} extends to bounded bilinear maps
	$$\mr S_1(G_1\times G_2,\varphi_{1,t}^{-1}\times\varphi_{2,t}^{-1})\times\big(\mc H_{\mr C_1}\otimes_{\mr h} \mc K_{\mr R_1}\big)\to \mc H_{\mr C_1}\otimes_{\mr h} \mc K_{\mr R_1}$$
	and
	$$\mr S_1(G_1\times G_2,\varphi_{1,t}^{-1}\times\varphi_{2,t}^{-1})\times\big(\mc H_{\mr C_\infty}\otimes_{\mr h} \mc K_{\mr R_\infty}\big)\to \mc H_{\mr C_\infty}\otimes_{\mr h} \mc K_{\mr R_\infty},$$
	Choose $\theta\in [0,1]$ so that $\frac{1}{q}=\frac{1-\theta}{1}+\frac{\theta}{2}$. Since
	$$ \mr S_q(G_1\times G_2,\varphi_{1,t}^{-1}\times \varphi_{2,t}^{-1})\cong[\mr S_1(G_1\times G_2,\varphi_{1,t}^{-1}\times \varphi_{2,t}^{-1}),\mr S_2(G_1\times G_2,\varphi_{1,t}^{-1}\times \varphi_{2,t}^{-1})]_\theta$$
	and
	$$ \HCq\otimes_{\mr h}\KRq\cong[\mc H_{\mr C_1}\otimes_{\mr h} \mc K_{\mr R_1},\mc H_{\mr C_2}\otimes_{\mr h} \mc K_{\mr R_2}]_\theta,$$
	it follows by bilinear interpolation (see \cite[Section 10.1]{bilinear-interpolation}) that Equation \eqref{eq:Sq-inclusion-lem} extends to a bounded bilinear map
	$$\mr S_q(G_1\times G_2,\varphi_{1,t}^{-1}\times\varphi_{2,t}^{-1})\times\big(\mc H_{\mr C_q}\otimes_{\mr h} \mc K_{\mr R_q}\big)\to \mc H_{\mr C_q}\otimes_{\mr h} \mc K_{\mr R_q},$$
	Similarly, since
	$$ \mr S_q(G_1\times G_2,\varphi_{1,t}^{-1}\times \varphi_{2,t}^{-1})\cong[\mr S_1(G_1\times G_2,\varphi_{1,t}^{-1}\times \varphi_{2,t}^{-1}),\mr S_2(G_1\times G_2,\varphi_{1,t}^{-1}\times \varphi_{2,t}^{-1})]_\theta$$
	and
	$$ \HCp\otimes_{\mr h}\KRp\cong[\mc H_{\mr C_\infty}\otimes_{\mr h} \mc K_{\mr R_\infty},\mc H_{\mr C_2}\otimes_{\mr h} \mc K_{\mr R_2}]_\theta,$$
	it follows that Equation \eqref{eq:Sq-inclusion-lem} extends to a bounded bilinear map
	$$\mr S_q(G_1\times G_2,\varphi_{1,t}^{-1}\times\varphi_{2,t}^{-1})\times\big(\mc H_{\mr C_p}\otimes_{\mr h} \mc K_{\mr R_p}\big)\to \mc H_{\mr C_p}\otimes_{\mr h} \mc K_{\mr R_p}.$$
	Therefore, we have a continuous embedding
	$$ \mr S_q(G_1\times G_2,\varphi_{1,t}^{-1}\times \varphi_{2,t}^{-1})\subset \mr C^*_{\mr r}(G_1){\opq} \mr C^*_{\mr r}(G_2).$$
\end{proof}

We now prove the main theorem of this section. Some classes of groups satisfying the hypotheses of the following theorem are listed in Example \ref{ex:groups}. We recall that a function $\mc L: G \to \mathbb{R}_+$ is a {\em conditionally negative definite} on a discrete group $G$ if the functions $\varphi_{i,t}\colon G_i\to (0,1]$  defined by
$$\varphi_{t}(s)=e^{-t\mc L_i(s)}  \ \ (s\in G),$$
are positive definite functions on $G$.

\begin{thm}\label{thm:group-main}
	Let $2\leq p<\infty$ and $q$ be the H\"older conjugate of $p$. Suppose $G_1$ and $G_2$ are discrete groups that possess the rapid decay property with respect to conditionally negative definite length functions $\mc L_1$ and $\mc L_2$. The following are equivalent for a positive definite function $\phi$ on $G_1\times G_2$.
	\begin{enumerate}
		\item $\phi\in \mr B_{\mr S_p}(G_1\times G_2)$;
		\item $\phi$ extends to a state of $\mr C^*_{\mr r}(G_1){\otimes_{\mr C^*_{p,q}}}\mr C^*_{\mr r}(G_2)$;
		\item $\phi\in \mr S_p(G_1\times G_2,\varphi_{1,t}\times \varphi_{2,t})$ for all $t>0$.
	\end{enumerate}
\end{thm}

\begin{proof}
	\textit{(1)} $\Rightarrow$ \textit{(2)}: Corollary \ref{cor:p,q-to-Sp}.
	
	{\emph{(2)} $\Rightarrow$ \emph{(3)}:} Lemma \ref{lem:Sq-inclusion} guarantees
	\begin{eqnarray*}
		&&\big[\mr C^*_{\mr r}(G_1){\otimes_{\mr C^*_{p,q}}} \mr C^*_{\mr r}(G_2)\big]^*\\
		&\subset& \big[\mr C^*_{\mr r}(G_1){\otimes_{p,q}} \mr C^*_{\mr r}(G_2)\big]^*\\
		&\subset& \big[\mr S_q(G_1\times G_2,\varphi_{2,t}^{-1}\times \varphi_{1,t}^{-1})\big]^*\\
		&=&\mr S_p(G_1\times G_2,\varphi_{1,t}\times \varphi_{2,t}).
	\end{eqnarray*}
	
	{\textit{(3)}} $\Rightarrow$ \textit{(1)}: Let $\varphi_t:=\varphi_{1,t}\times \varphi_{2,t}$. The hypothesis (condition \textit{(4)}) is equivalently restated as $\phi\varphi_t\in \mr S_p(G_1\times G_2)$ for all $t\in \R$. As $\phi\varphi_t$ is a positive definite function for all $t\in \R$, we deduce $\phi\varphi_t\in \mr B_{\mr S_p}(G_1\times G_2)$ for all $t\in \R$. Taking the weak* limit as $t\to 0^+$ allows us to conclude $\phi\in \mr B_{\mr S_p}(G_1\times G_2)$.
\end{proof}

\begin{cor}
	Let $2\leq p<\infty$ and $q$ be the H\"older conjugate of $p$. Suppose $G_1$ and $G_2$ are discrete groups that possess the rapid decay property with respect to conditionally negative definite length functions $\mc L_1$ and $\mc L_2$. {Then
	{$$\mr C^*_{\mr S_p}(G_1\times G_2)\cong\mr C^*_{\mr r}(G_1)\otimes_{\mr C^*_{p,q}} \mr C^*_{\mr r}(G_2)$$}
	canonically.}
\end{cor}

As a consequence of the previous theorem, we show the pairwise distinctness of the C*-algebras $\mr C^*_{\mr r}(G){\otimes_{\mr C^*_{p,q}}}\mr C^*_{\mr r}(G)$ for $p\in [2,\infty]$ when $G$ is taken from a large class of non-amenable discrete groups possessing both the Haagerup property and the rapid decay property.

\begin{cor}\label{cor:different-red-group}
	Suppose $G$ is non-amenable, discrete group that has the rapid decay property with respect to a conditionally negative definite function $\mc L$ so that $\varphi_t\colon G\to \R$ defined by $\varphi_t(s)=e^{-t\mc L(s)}$ belongs to $\bigcup_{1\leq p <\infty} \mr \ell^p(G)$ for $t\in (0,\infty)$. The canonical map
	$$ \mr C^*_{\mr r}(G){\otimes_{\mr C^*_{p,q}}}\mr C^*_{\mr r}(G)\to \mr C^*_{\mr r}(G){\otimes_{\mr C^*_{p',q'}}}\mr C^*_{\mr r}(G)$$
	is not injective for $2\leq p'<p\leq \infty$.
\end{cor}

\begin{proof}
	Let
	$$t_0=\inf\{ t>0 : \varphi_t\in \ell^1(G)\}.$$
	Then $t_0>0$ and if $p>0$, then $\varphi_{t}\in \ell^p(G)$ for $t>\frac{t_0}{p}$ and $\varphi_{t}\not\in \ell^p(G)$ for $0<t<\frac{t_0}{p}$.
	Defining $\widetilde{\varphi}_t\colon G\times G\to \C$ to be given by
	$$ \widetilde{\varphi}_t(s_1,s_2)=\left\{\begin{array}{c l}
	\varphi_t(s_1), & \text{if }s_1=s_2\\
	0 & \text{if }s_1\neq s_2
	\end{array}\right. ,$$
	we deduce that $\widetilde{\varphi_t}\in \mr S_p(G_1\times G_2)$ if $t>\frac{t_0}{p}$, and $\widetilde{\varphi_t}\not\in \mr S_p(G_1\times G_2)$ if $0<t<\frac{t_0}{p}$. Observing that $\widetilde{\varphi_t}$ is a positive definite function and
	$$\widetilde{\varphi_t}(\varphi_{t'}\times \varphi_{t'})=\widetilde{\varphi_{t+2t'}}$$
	for $t,t'>0$, it follows from Theorem \ref{thm:group-main} that
	$$ \widetilde{\varphi_t}\in \mr B_{\mr S_p}(G_1\times G_2)$$
	if and only if $t\geq \frac{t_0}{p}$.
\end{proof}

There are many interesting discrete groups that satisfy the hypotheses of Corollary \ref{cor:different-red-group}. These examples are known to experts and have implicitly appeared elsewhere in the literature.
\begin{example}[see {\cite{SW-exotic}}]\label{ex:groups}\
	\begin{enumerate}
		\item If $G$ is either a finitely generated free group or a finitely generated Coxeter group, then the canonical word length function {$|\cdot |$} on $G$ is conditionally negative definite and the function $\varphi_t\colon G\to \R$ defined by $\varphi_t(s)=e^{-t|s|}$ belongs to $\bigcup_{1\leq p <\infty}\ell^p(G)$ for every $t>0$.
		\item Suppose $G$ is a discrete group acting properly and cocompactly by isometries on a CAT(0) cube complex $X$ with the isolated flats property. Fix a vertex $x_0$ of $X$ and consider the combinatial distance $d_c$ on the $1$-skeleton of $X$. The function $\mc L\colon G\to [0,\infty)$ given by $\mc L(s)=d_c(sx_0,x_0)$ is a conditionally negative definite length function such that the function $\varphi_t\colon G\to \R$ given by $\varphi_t(s)=e^{-t\mc L(s)}$ belongs to $\bigcup_{1\leq p <\infty}\ell^p(G)$ for every $t>0$.
	\end{enumerate}
\end{example}

\section{Failure of projectivity and injectivity for enveloping C*-algebras}

It is well known that the minimal tensor product $\otimes_{\min}$ is injective in the sense that if $A$ and $B$ are C*-algebras and $C$ is a C*-subalgebra of $A$, then $C\otimes_{\min} B$ is a C*-subalgebra of $A\otimes_{\min}B$, and the maximal tensor product $\otimes_{\max}$ is projective in the sense that if $A$ and $B$ are C*-algebras and $J$ is a two-sided closed ideal of $A$, then $(A\otimes_{\max} B)/(J\otimes_{\max}B)$ is canonically $*$-isomorphic to $(A/J)\otimes_{\max} B$. It is natural to wonder whether the tensor product functor {$\otimes_{\mr C^*_{p,q}}$} is either injective or projective for H\"older conjugate $p,q\in [1,\infty]$ with $p\neq 1,2,\infty$. Using the results from the previous section along with the fact that the free group $\F_2$ possesses the factorization property, we can definitively show that this is not the case.

Recall first that a discrete group $G$ has the \emph{factorization property} (see \cite[Section 6.4]{BO}) if the representation $\lambda\cdot\rho \colon G\times G\to \mr B(\ell^2(G))$ given by $(\lambda\cdot \rho)(s,t)=\lambda(s)\rho(t)$, where $\rho$ denotes the right regular representation of $G$, extends to a $*$-representation of $\mr C^*(G)\otimes_{\min} \mr C^*(G)$. Since the characteristic function $1_\Delta$ of the diagonal subgroup $\Delta=\text{diag}(G\times G)$ is a positive definite function of $G\times G$ with GNS representation $\lambda\cdot \rho$, possessing the factorization property is equivalent to requiring that $1_\Delta$ extends to a state on $\mr C^*(G)\otimes_{\min} \mr C^*(G)$.

\begin{thm}\
	\begin{enumerate}
		\item If $2<p\leq\infty$ and $q$ is the H\"older conjugate of $p$, then the canonical map $\mr C^*(\F_2){\otimes_{\mr C^*_{p,q}}}\mr C^*_{\mr r}(\F_2)\to \mr C^*(\F_2){\otimes_{\mr C^*_{p,q}}}\mr B(\ell^2(\F_2))$ is not injective.
		\item Suppose $2\leq p<\infty$, $q\colon \mr C^*(\F_2)\to \mr C^*_{\mr r}(\F_2)$ is the canonical map, and $J$ is the norm closure of $(\mr C^*(\F_2)\odot \ker q)+(\ker q\odot \mr C^*(\F_2))$ inside of $\mr C^*(\F_2){\otimes_{\mr C^*_{p,q}}}\mr C^*(\F_2)$. The canonical map $[\mr C^*(\F_2){\otimes_{\mr C^*_{p,q}}}\mr C^*(\F_2)]/J\to \mr C^*_{\mr r}(\F_2){\otimes_{\mr C^*_{p,q}}}\mr C^*_{\mr r}(\F_2)$ is not injective.
	\end{enumerate}
\end{thm}

\begin{proof}
	\emph{(1)}: Since $\mr C^*(\F_2)\otimes_{\max}\mr B(\ell^2(\F_2))\cong\mr C^*(\F_2)\otimes_{\min}\mr B(\ell^2(\F_2))$ canonically (see \cite[Theorem 13.2.1]{BO}), it suffices to show that the canonical map $\mr C^*(\F_2)\otimes_{\mr C^*_{p,q}}\mr C^*_{\mr r}(\F_2)\to \mr C^*(\F_2)\otimes_{\min}\mr C^*_{\mr r}(\F_2)$ is not injective.
	
	For each $t>0$, let $\varphi_t\colon \F_2\to \F_2$ be the positive definite function given by $\varphi_t(s)=e^{-t|s|}$ for $s\in \F_2$. Then $\varphi_t\in \ell^p(\F_2)$ if and only if $t>3^{-1/p}$. Choose $t_0$ such that $3^{-1/p}<t_0<3^{-1/2}$ and consider the positive definite function $\psi\colon \F_2\times \F_2\to \C$ given by $\psi(s_1,s_2)=0$ if $s_1\neq s_2$, and $\psi(s,s)=\varphi_{t_0}(s)$. Then $\psi\in \mr S_p(\F_2\times \F_2)$ since $\psi$ is $\ell^p$-summable and supported on $\text{diag}(\F_2\times \F_2)$. Hence, $\psi\in \mr B_{\mr S_p}(\F_2\times \F_2)$. Since $\mr C^*(\F_2){\otimes_{\mr C^*_{p,q}}}\mr C^*_{\mr r}(\F_2)$ quotients onto $\mr C^*_{\mr r}(\F_2){\otimes_{\mr C^*_{p,q}}}\mr C^*_{\mr r}(\F_2)$, we deduce that $\psi$ extends to a state on $\mr C^*(\F_2){\otimes_{\mr C^*_{p,q}}}\mr C^*_{\mr r}(\F_2)$ by Theorem \ref{thm:group-main}. We next show that $\psi$ does not extend to a state on $\mr C^*(\F_2)\otimes_{\min}\mr C^*_{\mr r}(\F_2)$.
	
	Let $\Delta=\mr{diag}(\F_2\times \F_2)$, and $\pi_u$ be the universal representation of $\F_2$. If $x\in \mr{c_c}(\Delta)\subset \mr{c_c}(\F_2\times \F_2)$, then
	$$\|x\|_{\mr C^*(\F_2)\otimes_{\min}\mr C^*_{\mr r}(\F_2)}=\|(\pi_u\times \lambda)(x)\|=\|\lambda(x)\|=\|x\|_{\mr C^*_{\mr r}(\Delta)},$$
	since $(\pi_u\times \lambda)|_\Delta=\pi_u\otimes \lambda$ is unitarily equivalent to an amplification of the left regular representation by the Fell absorption principle. It follows from \cite[Corollay 3.5]{Okayasu} that the positive definite function $\psi|_\Delta$ does not extend to a state of $\mr C^*_{\mr r}(\Delta)$ and, hence, that $\psi$ does not extend to a state of $\mr C^*(\F_2)\otimes_{\min}\mr C^*_{\mr r}(\F_2)$.
	
	\emph{(2)}: Since $\F_2$ has the factorization property, the representation $\lambda\cdot \rho\colon \F_2\times \F_2\to \mr B(\ell^2(\F_2))$ extends to a C*-representation of $\mr C^*(\F_2)\otimes_{\min}\mr C^*(\F_2)$ and, hence, also of $\mr C^*(\F_2){\otimes_{\mr C^*_{p,q}}}\mr C^*(\F_2)$. As $\lambda\cdot \rho|_{\F_2\times \{e\}}$ is the left regular representation of $\F_2$ and $\lambda\cdot \rho|_{\{e\}\times \F_2}$ is unitarily equivalent to the left regular representation of $\F_2$, we deduce that the kernel of $\lambda\cdot \rho$ contains $J$ when $\lambda\cdot \rho$ is viewed as a $*$-representation of $\mr C^*(\F_2){\otimes_{\mr C^*_{p,q}}}\mr C^*(\F_2)$. Hence, $1_\Delta$ extends to a state of $[\mr C^*(\F_2)\otimes\mr C^*(\F_2)]/J$. However, $1_\Delta$ does not extend to a state of $\mr C_{\mr r}^*(\F_2){\otimes_{\mr C^*_{p,q}}}\mr C^*_{\mr r}(\F_2)$ by Theorem \ref{thm:group-main} since $1_\Delta(\varphi_t\times \varphi_t)\not\in \mr S_p(\F_2\times \F_2)$ for $0<t<\frac{3^{-1/p}}{2}$.
\end{proof}

\begin{rem}
	The $2^{\aleph_0}$ tensor product functors constructed by Ozawa and Pisier in \cite{OP} are each injective. Hence, if $p,q\in [1,\infty]$ are H\"older conjugate with $p\neq 2$, then the tensor product functor ${\otimes_{\mr C^*_{p,q}}}$ differs from each of the Ozawa-Pisier tensor product functors.
\end{rem}

\section{Rigidly symmetric C*-algebras}

In this section, we prove the converse to Kugler's theorem stating that every type I C*-algebra is rigidly symmetric. Our proof relies heavily on the constructions we introduced in this paper and uses the results on tensor products of discrete group C*-algebras from Section 5. Another key ingredient in our proof will be the concept of spectral interpolation introduced in \cite{SW-Hermitian}, which we recall below.


Generalizing the concept of a symmetric $*$-algebra, a $*$-subalgebra $S$ of a $*$-algebra $A$ is \emph{quasi-symmetric in $A$} if $\mr{Sp}_A(a^*a)\subset [0,\infty)$ for all $a\in S$. We will be most interested in the case when $A$ is a Banach $*$-algebra and $S$ is a dense $*$-subalgebra of $A$, but it will be convenient at times for us to consider these concepts more generally.

Suppose $A,B,C$ are Banach $*$-algebras with continuous, dense inclusions $A\subset B\subset C$. The triple $(A,B,C)$ is a \emph{spectral interpolation of triple Banach $*$-algebras} with respect to a dense $*$-subalgebra $S\subset A$ if there exists $\theta\in (0,1)$ so that the spectral radius relation
$$ \mr r_B(a)\leq \mr r_A(a)^{1-\theta}\mr r_C(a)^\theta$$
holds for every $a=a^*\in S$. {Here, $\mr r_B$ refers to the spectral radius with respect to the algebra $B$.}

The primary example of a spectral interpolation of Banach $*$-algebras that we will consider in this paper is given by the following proposition.

\begin{prop}
	{Suppose $A$ and $B$ are C*-algebras and $1\leq p_0< p_1<p_2\leq 2$. Let $q_0,q_1,q_2$ be the H\"older conjugates of $p_0,p_1,p_2$, respectively. The triple $(A\otimes_{p_0,q_0}B,A\otimes_{p_1,q_1}B,A\otimes_{p_2,q_2}B)$ is a spectral interpolation of triple Banach $*$-algebras with respect to $A\odot B$.}
\end{prop}

\begin{proof}
	This follows from Remark \ref{rem:norm-relation} and the spectral radius formula.
\end{proof}

We will require the following result proven by the second and third author.

\begin{thm}[{S.-W. \cite[Theorem 3.4]{SW-exotic}}]\label{thm:SW}
	Suppose $(A,B,C)$ is a spectral interpolation of $*$-semisimple Banach $*$-algebras with respect to $S\subset A$. If $S$ is quasi-symmetric in $A$, then $\mr r_B(x)=\mr r_C(x)$ for all self adjoint $x\in S$.
\end{thm}

The above theorem is stated more generally for whenever $S$ is ``quasi-Hermitian in $A$'', however we will only use it in the case when $S$ is quasi-symmetric in $A$  (see \cite{SW-Hermitian} for relevant definition and the implication of quasi-symmetric implying quasi-Hermitian).

The following lemma almost does not need to be stated, but is included for the sake of completeness.
\begin{lem}\label{lem:quasi-obvious}
	Let $A$ be a Banach $*$-algebra and $S$ be a $*$-subalgebra of $A$. Suppose that $S$ is quasi-symmetric in $A$.
	\begin{enumerate}
		\item If $A_0$ is a norm-closed $*$-subalgebra of $A$, then $A_0\cap S$ is quasi-symmetric in $A_0$.
		\item If $B$ is a Banach $*$-algebra and $\pi\colon A\to B$ is a bounded $*$-homomorphism with dense range, then $\pi(S)$ is quasi-symmetric in $B$.
	\end{enumerate}
\end{lem}

\begin{proof}
	\emph{(1)} Suppose $a\in A_0\cap S$ is self-adjoint. Then $\partial\mr{Sp}_{A_0}(a^*a)$ (the boundary of $\mr{Sp}_{A_0}(x)$) is contained in $\mr{Sp}_{A}(a^*a)$, which is contained in $[0,\infty)$. It follows that $A_0\cap S$ is quasi-symmetric in $A_0$.
	
	\emph{(2)} Suppose $b\in \pi(S)$ and choose $a\in S$ such that $\pi(a)=b$. Then $\pi(a^*a)=b^*b$. Since
	$ \mr{Sp}_B(b^*b)\subset \mr{Sp}_A(a^*a)\subset [0,\infty),$
	we conclude that $\pi(S)$ is quasi-symmetric in $B$.
\end{proof}

The following lemma comprises the main part of the argument characterizing rigidly symmetric C*-algebras.

\begin{lem}\label{lem:Hem-main-lem}
	Let $A$ and $B$ be C*-algebras and $1\leq p<q\leq \infty$ be H\"older conjugate. Consider the following conditions.
	\begin{enumerate}
		\item $A\odot B$ is quasi-symmetric in $A{\opq} B$;
		\item If $C$ and $D$ are quotients of C*-subalgebras of $A$ and $B$, respectively, then $\mr r_{C{\opq} D}(x^*x)=\mr r_{C\otimes_{\min}D}(x^*x)$ for every $x\in C\odot D$.
		\item If $C$ and $D$ are quotients of C*-subalgebras of $A$ and $B$, respectively, then $C{\otimes_{\mr C^*_{p,q}}} D\cong C\otimes_{\min}D$ canonically.
	\end{enumerate}
	Then (1) $\Rightarrow$ (2) $\Rightarrow$ (3).
\end{lem}

\begin{proof}
	\textit{(1)} $\Rightarrow$ \textit{(2)}: Suppose $A_0$ and $B_0$ are C*-subalgebras of $A$ and $B$ such that $C$ is a quotient of $A_0$ and $D$ is a quotient of $B_0$. Since the bifunctor $(A,B)\mapsto A{\opq} B$ is injective and isometrically commutative (see {Proposition \ref{prop:containment}} and Proposition \ref{prop:commutative}), it follows from Lemma \ref{lem:quasi-obvious}(1) that $A_0\odot B_0$ is quasi-symmetric in $A_0{\opq} B_0$. Further, since the canonical map $A_0{\opq} B_0\to C{\opq} D$ is a contractive $*$-homomorphism by Proposition \ref{prop:functorial}, it follows by Lemma \ref{lem:quasi-obvious}(2) that $C\odot D$ is quasi-symmetric in $C{\opq} D$.
	
	{Consider concrete embeddings $C\subset \mr B(\mc H)$ and $D\subset \mr B(\mc K)$ of $C$ and $D$. Replacing $\mc H$ and $\mc K$ by their infinite amplifications, we may assume that
	$$ \|\pi_s(y)\|_\mr{op}=\|y\|_{C\otimes_{s} D}$$
	for all $s\in [1,\infty]$ and $y\in C\odot D$ by Lemma \ref{lem:easier}.}
	Let $x=\sum_{i=1}^k a_i\otimes b_i\in C\odot D$ and
	suppose $F\in \mr S_p(\overline{\mc K}, \mc H)$ has finite rank $m$ and $\|F\|_{\mr S_p}=1$. Setting $P_n\in \mr B(\mc H)$ to be the orthogonal projection onto the range of $\pi_p(x^*x)^nF$, we have $\mr{rank}\,P_n\leq mk^{2n}$ for each $n\in \N$. As such, if $n\in \N$ and $p<p'\leq 2<q''<\infty$ satisfy $\frac{1}{p'}+\frac{1}{q''}=\frac{1}{p}$, then
	\begin{eqnarray*}
		\| \pi_p(x^*x)^nF\|_{\mr S_p} &=& \|P_n\pi_p(x^*x)^nF\|_{\mr S_p}\\
		&\leq& \|P_n\|_{\mr S_{q''}}\|\pi_{p'}(x^*x)^nF\|_{\mr S_{p'}}\\
		&\leq& (mk^{2n})^{\frac{1}{q''}}\|\pi_{p'}(x^*x)^n\|_{\op}\\
		&\leq& (mk^{2n})^{\frac{1}{q''}}\|(x^*x)^n\|_{C{\otimes_{p',q'}}D},
	\end{eqnarray*}
	where $q'$ is the H\"older conjugate of $p'$.
	Since $\|\pi_p(x^*x)^n\|=\|\pi_q(x^*x)^n\|$ by virtue of $(x^*x)^n\in C\odot D$ being self adjoint (see Corollary \ref{cor:*-isometric}), we deduce
	$$ \|(x^*x)^n\|_{C{\otimes_{p,q}}D}\leq (mk^{2n})^{\frac{1}{q''}}\|(x^*x)^n\|_{C{\otimes_{p',q'}} D}$$
	for all $n\in \N$. Hence,
	\begin{eqnarray}
	r_{C{\otimes_{p,q}}D}(x^*x)&=& \lim_{n\to \infty} \|(x^*x)^n\|_{C{\opq} D}^{\frac{1}{n}} \nonumber \\
	&\leq& \liminf_{n\to \infty} \left((mk^{2n})^{\frac{1}{q''}}\|(x^*x)^n\|_{C{\otimes_{p',q'}} D}\right)^{\frac{1}{n}} \nonumber\\
	&=& k^{\frac{2}{q''}}\mr r_{C{\otimes_{p',q'}} D}(x^*x) \nonumber \\
	&=& k^{\frac{2}{q''}} \mr r_{C\otimes_{\min} D}(x^*x).
	\end{eqnarray}
	The final line in this equation follows from the preceding by Theorem \ref{thm:SW} since $(C{\otimes_{p,q}}D, C{\otimes_{p',q'}} D, C\otimes_{\min} D)$ is a spectral interpolations of triple Banach $*$-algebras and $C\odot D$ is quasi-{symmetric} in $C{\opq} D$. As this holds for every $q''\in (2,\infty)$, we deduce
	$$ r_{C{\opq} D}(x^*x)\leq r_{C\otimes_{\min} D}(x^*x)$$
	by taking the limit as $q''\to\infty$.
	
	\textit{(2)} $\Rightarrow$ \textit{(3)}: Suppose $x\in C\odot D$. Then
	$$ \|x\|_{C{\otimes_{\mr C^*_{p,q}}}D}^2=\mr r_{C{\otimes_{\mr C^*_{p,q}}}D}(x^*x)\leq \mr r_{C{\opq} D}(x^*x)=\mr r_{C\otimes_{\min}D}(x^*x)=\|x\|_{\min}^2.$$
	It follows that $\|x\|_{C{\otimes_{\mr C^*_{p,q}}}D}=\|x\|_{\min}$ for all $x\in C\odot D$.
\end{proof}

Before proving the next theorem, we need a technical lemma that will allow us to pass to unitizations of C*-algebras.

\begin{lem}\label{lem:unitizations}
	Suppose $A$ and $B$ be are C*-algebras, and let $\widetilde{A}$ and $\widetilde{B}$ denote the unitizations of $A$ and $B$. If $A\odot B$ is quasi-symmetric $A{\opq} B$ for H\"older conjugate $p,q\in [1,\infty]$, then $\widetilde{A}\odot \widetilde{B}$ is quasi-symmetric in $\widetilde A{\opq} \widetilde B$.
\end{lem}

\begin{proof}
	We will prove this in the case where neither $A$ nor $B$ is unital. The other cases are deduced similarly.
	
	It is known that if a $*$-algebra $C$ admits a $*$-ideal $J$ such that both $J$ and $C/J$ are symmetric, then $C$ itself must be symmetric (see \cite[Theorem 9.8.3 (e)$\Rightarrow$(a)]{palmer}). The proof of this result also shows the following: Suppose $J$ is a $*$-ideal of a $*$-algebra $C$ with $J$ and let $q\colon C\to C/J$ denote the quotient map. If $S$ is a $*$-subalgebra of $C$ such that $q(S)$ is quasi-symmetric in $C/J$ and $S\cap J$ is quasi-symmetric in $C$, then $S$ is quasi-symmetric in $C$.
	
	Let $C=\widetilde A{\opq} \widetilde B$, $S=\widetilde A\odot \widetilde B$ and $J=A{\opq} B$. Since $S\cap J=A\odot B$ is quasi-symmetric in $A{\opq} B$ and, hence, also in $C$  by assumption, and $q(S)$ is symmetric by virtue of being $*$-isomorphic to the C*-algebra $\C\oplus A\oplus B$, we deduce that $S=\widetilde A\odot \widetilde B$ is quasi-symmetric in $\widetilde A{\opq} \widetilde B$.
\end{proof}

Alas, we are now equipped to prove the main theorem of this section.

\begin{thm}\label{thm:main}
	Let $A$ be a C*-algebra, and $\{p,q\}\subseteq [1,\infty]\backslash\{2\}$ be H\"older conjugate. The following are equivalent.
	\begin{enumerate}
		\item $A$ is type I;
		\item If $B$ is a symmetric Banach $*$-algebra, then so is $A\otimes_{\gamma} B$;
		\item $A$ is rigidly symmetric;
		\item There exists a non-type I C*-algebra $B$ such that $A\odot B$ is quasi-symmetric in $A{\opq} B$.
	\end{enumerate}
\end{thm}

\begin{proof}
	\emph{(1)} $\Rightarrow$ \emph{(2)}: This was proven by Kugler in \cite{Kugler}.
	
	\emph{(2)} $\Rightarrow$ \emph{(3)}: Trivial.
	
	\emph{(3)} $\Rightarrow$ \emph{(4)}: Trivial.
	
	
%
%
		
	\emph{(4)} $\Rightarrow$ \emph{(1)}: Let $\widetilde A$ and $\widetilde B$ denote the unizations of $A$ and $B$ respectively and suppose that $A$ and $B$ are each not of type I. Then $\widetilde A$ and $\widetilde{B}$ are not of type I. So, as shown by Blackadar in \cite{Blackadar}, $\widetilde A$ and $\widetilde B$ each contain a C*-subalgebra that quotients onto the Choi algebra $\mr C^*_{\mr r}(\Z_2*\Z_3)$. Since $\Z_2*\Z_3$ contains $\F_2$ as a subgroup, we deduce that $\mr C^*_{\mr r}(\F_2)$ is the quotient of a C*-subalgebra of each of $A$ and $B$. By Corollary \ref{cor:different-red-group}, we know that the canonical map $\mr C^*_{\mr r}(\mathbb F_2){\otimes_{\mr C^*_{p,q}}}\mr C^*_{\mr r}(\mathbb F_2)\to \mr C^*_{\mr r}(\mathbb F_2)\otimes_{\min} \mr C^*_{\mr r}(\mathbb F_2)$ is non-injective. It follows from Lemma \ref{lem:Hem-main-lem} that $\widetilde A\odot \widetilde B$ is not quasi-symmetric in $\widetilde A{\opq}\widetilde B$. So $A\odot B$ is not quasi-symmetric in $A{\opq} B$ by Lemma \ref{lem:unitizations}.
\end{proof}

Consequently, we deduce the following.

\begin{cor}\label{cor:Herm-main}
	If $A$ and $B$ are C*-algebras, then $A\otimes_{\gamma} B$ is symmetric if and only if $A$ or $B$ is type I.
\end{cor}

One might be tempted to conjecture in light of the proof of Theorem \ref{thm:main} that if $A$ is a rigidly symmetric Banach $*$-algebra, then its enveloping C*-algebra $\mr C^*(A)$ must be rigidly symmetric and, hence, type I. However, the proof fails to demonstrate this since if $A$ is a Banach $*$-algebra and $B$ is a norm-closed $*$-subalgebra of $\mr C^*(A)$, then it need not be the case that $A\cap B$ is norm-dense in $B$. Moreover, the following remark gives a counterexample to this conjecture.

Recall that a locally compact group $G$ with Haar measure $\mu$ is said to have \textit{polynomial growth} if for every compact subset $K$ of $G$, there exists $k\in\N$ and $C>0$ such that $\mu(K^n)\leq Cn^k$ for all $n\in \N$. Similarly, $G$ is of \textit{subexponential growth} if for all $t>0$, there exists $C>0$ so that $\mu(K^n)\leq Ce^{tn}$ for all $n\in \N$. We will make use of the following theorem.

\begin{thm}[Leptin-Poguntke \cite{LP}]\label{thm:FGL}
	If $G$ is a finitely generated group of polynomial growth, then $\ell^1(G)$ is rigidly symmetric.
\end{thm}

\begin{rem}
	Let $G$ be a finitely generated discrete group that is of polynomial growth but does not contain an abelian group of finite index (e.g., take $G$ to be the discrete Heisenberg group $\mr H_3(\Z)$). Then $\ell^1(G)$ is rigidly symmetric by Theorem \ref{thm:FGL}. However, $\mr C^*(G)=\mr C^*(\ell^1(G))$ is not rigidly symmetric since $G$ does not contain a finite index abelian subgroup by Thoma's characterization of discrete groups with Type I C*-algebras (see \cite{Thoma}). In particular, the enveloping C*-algebra of a rigidly symmetric Banach $*$-algebra need not be rigidly symmetric.
\end{rem}

Though the group algebra $\mr L^1(G)$ of a locally compact group with subexponential growth need not be symmetric, it was shown by Palma that $\mr{C_c}(G)$ is quasi-symmetric in $\mr L^1(G)$ for every locally compact group with subexponential growth (see \cite{Palma}). We finish this section by refining this result of Palma into a statement more analogous to Theorem \ref{thm:FGL}.

\begin{prop}
	Let $G$ be a locally compact group of subexponential growth and $A$ a C*-algebra. The $*$-algebra $\mr{C_c}(G,A)$ is quasi-symmetric in $\mr L^1(G,A)\cong \mr L^1(G)\otimes_{\gamma} A$.
\end{prop}

\begin{proof}
	 Suppose that $A$ is an arbitrary C$^*$-algebra. Without loss of generality, we may assume that $A$ is unital since $\mr L^1(G,A)$ embeds isometrically inside of $\mr L^1(G,\widetilde A)$ where $\widetilde A$ denotes the unitization of $A$. Take $f\in \mr{C_c}(G,A)$ to be a self-adjoint element, and let $S$ be a pre-compact, open, symmetric subset of $G$ with $\text{supp}\,f \subseteq S$. We will write $f^n$ to denote $\underbrace{f*\cdots*f}_{n\text{ times}}$. Then for every $n\in \N$ we have
	\begin{eqnarray*}
		\|f^{n}\|_{\mr L^1(G,A)} &=& \|f^{n} (1_{S^n}\otimes 1) \|_{\mr L^1(G,A)} \\
		&\leq & \|f^{n}\|_{\mr L^2(G,A)} \|1_{S^n}\otimes 1\|_{\mr L^2(G,A)}  \\
		&\leq &  \|f^{n-1}\|_{\mr B(\mr L^2(G,A))}\|f\|_{\mr L^2(G,A)} \mu(S^n)^{\frac{1}{2}}
		\\
		&=&  \|f^{n-1}\|_{\mr C^*_{\mr r}(G)\otimes_{\min} A}\|f\|_{\mr L^2(G,A)} \mu(S^n)^{\frac{1}{2}},
	\end{eqnarray*}
	where $\mu$ denotes the Haar measure of $G$. Here we use the fact that the norm of the minimal C*-tensor algebra $\mr C^*_{\mr r}(G)\otimes_{\min} A$ is the same as the one coming from canonical convolution of elements of $\mr L^1(G,A)$ on the Hilbert C*-module $\mr L^2(G,A)$. Hence, taking the $n^{\text{th}}$ root of both sides as $n$ tends towards infinity implies
	\begin{equation*}
		\mr r_{\mr L^1(G,A)}(f)\leq  \mr r_{\mr C^*_{\mr r}(G)\otimes_{\min} A}(f)\cdot\lim_{n\to \infty}\mu(S^n)^{\frac{1}{2n}}=\mr r_{\mr C^*_{\mr r}(G)\otimes_{\min} A}(f),
	\end{equation*}
	as $G$ has subexponential growth so that $\lim_{n\to \infty} \mu(S^n)^{\frac{1}{2n}}=1$. Thus $\mr r_{\mr L^1(G,A)}(f)=\mr r_{\mr C^*_{\mr r}(G)\otimes_{\min} A}(f)$
	for all $f^*=f\in \mr{C_c}(G,A)$. It now follows from the Barnes-Hulanicki Theorem (see \cite{Barnes}) that $$\mr{Sp}_{\mr L^1(G,A)}(f)=\mr{Sp}_{\mr C^*_{\mr r}(G)\otimes_{\min}A}(f)$$
	for all $f\in \mr{C_c}(G,A)$. So $\mr{C_c}(G,A)$ is quasi-symmetric in $\mr L^1(G,A)$.
\end{proof}

\section{Problems}

We finish off this paper by listing a few open problems. The first question is inspired by the fact that if $G$ is a locally compact group such that there exists $p\in (1,\infty)$ where the canonical quotient map $\mr C^*(G)\to\mr C^*(\mr{PF}^*_p(G))$ is injective, then $G$ is amenable (see \cite{SW-exotic}). The authors wonder whether the analogue of this result is true for C*-algebras.

\begin{question}\label{ques:1}
	Suppose $A$ and $B$ are C*-algebras such that there exist a H\"older conjugate pair $p,q\in (1,\infty)$ such that the $A\otimes_{\max}B\cong A\otimes_{\mr C^*_{p,q}}B$ canonically. Is it necessarily the case that $(A,B)$ is a nuclear pair, i.e.  $A\otimes_{\max}B\cong A\otimes_{\min} B$ canonically?
\end{question}

We also wonder whether results analogous to those obtained by Ozawa and Pisier in \cite{OP} are true for the bifunctors $(A,B)\mapsto A{\otimes_{\mr C^*_{p,q}}}B$.

\begin{question}\label{ques:2}
	Let $M$ and $N$ be nonnuclear von Neumann algebras. Is the canonical quotient map
	$M{\otimes_{\mr C^*_{p,q}}}N\to M{\otimes_{\mr C^*_{p',q'}}}N$
	non-injective for all $2\leq p'<p\leq \infty$. How about in the case when $M=N=\mr B(\ell^2(\N))$?
\end{question}

Lastly, we wonder whether the  bifunctors $(A,B)\mapsto A{\otimes_{\mr C^*_{p,q}}}B$ are functorial with respect to completely positive maps.

\begin{question}\label{ques:3}
	If $A,A',B,B'$ are C*-algebras, and $\varphi\colon A\to A'$ and $\psi\colon B\to B'$ are completely positive maps, is it necessarily the case that $\phi\otimes \psi\colon A\odot B\to A'\odot B'$ extends to a completely positive map
	$ A{\otimes_{\mr C^*_{p,q}}}B\to A'{\otimes_{\mr C^*_{p',q'}}}B'$?
\end{question}

\subsection*{Acknowledgements}
The authors are grateful to N. Ozawa, L. Turowska and an anonymous referee who independently suggested that one could use Voiculescu's theorem to show that the definition of $A\otimes_p B$ (in Definition \ref{D:independent}) is independent of the choice of the concrete representations of C*-algebras $A$ and $B$. This fact, proved in details in Theorem \ref{thm:independence}, answers a question asked in an earlier version of this paper.

\end{document}